\documentclass[12pt]{amsart}
\usepackage{amssymb}

\usepackage{latexsym}
\usepackage{amsmath}
\usepackage{amsfonts}
\usepackage{eufrak}

\newtheorem{definition}{Definition}
\newtheorem{theorem}{Theorem}
\newtheorem{proposition}{Proposition}
\newtheorem{lemma}{Lemma}

\def\.{\cdot}

\def\t{\tau}

\def\li{L_{\t_{i}}}
\def\li_p{L_{\t_{i+1}}}
\def\li_m{L_{\t_{i-1}}}
\def\lj{\tilde{L}_{\t_{j}}}
\def\lj_p{\tilde{L}_{\tilde{\t}_{j+1}}}
\def\lj_m{\tilde{L}_{\tilde{\t}_{j-1}}}
\def\li_plus{L_{\t\,_i,\t_{i+1}}}
\def\li_min{L_{\t\,_{i-1}},\t_i}
\def\lj_plus{\tilde{L}_{\tilde{\t}_j,\t_{j+1}}}
\def\lj_min{\tilde{L}_{\tilde{\t}_{j-1}},\t_j}

\def\Br{\text{Br}}
\def\Ir{\text{Ir}}

\def\local{\mathcal{\ell}}

\begin{document}

\title{Strong disorder in semidirected random polymers.}  
\author{N. Zygouras}   
\address{Department of Statistics\\
University of Warwick\\
Coventry CV4 7AL, UK.
e-mail:  \rm \texttt{N.Zygouras@warwick.ac.uk}}
\subjclass[2000]{ 60xx}
\keywords{Random walks, random potential, Lyapunov norms, strong disorder, localization, fractional moments}
\date{\today}
\maketitle
\begin{abstract}
We consider a random walk in a random potential, which models a situation of a random polymer and we study the annealed and quenched costs to perform long crossings from a point to a hyperplane. These costs are measured by the so called Lyapounov norms. We identify situations where the point-to-hyperplane annealed and quenched Lyapounov norms are different. We also prove that in these cases the polymer path exhibits localization.
\end{abstract}

\section{Introduction}
In the probabilistic literature polymers are modeled by a simple random walk $(X_n)_{n\geq 1}$ on $\mathbb{Z}^{d}$, $d\geq 1$. We denote by $P_x$ the distribution of the random walk, when it starts from $x\in \mathbb{Z}^{d}$. When the starting point coincides with the origin we will simply denote its distribution by 
$P$. We also consider a collection of i.i.d. random variables $(\,\omega(x))_{x\in\mathbb{Z}^{d}}$, independent of the walk.
We denote by $\mathbb{P}$ the distribution of this collection. We assume that $\omega$ is nonnegative, does not concentrate on a single point and that $\mathbb{E}[\omega^2]<\infty$.The polymer $(X_n)_{n\geq 1}$ interacts with the disorder $(\,\omega(x))_{x\in\mathbb{Z}^{d}}$, thus giving rise to the modeling of {\it random polymers}. This interaction can be modeled in a number of different ways, corresponding to various physical considerations. In this work we consider the case where the distribution of the random polymer is described in the following way: Let $\hat{l}\in \mathbb{Z}^{d}$ a unit vector, which plays the role of the direction and $T_L^{\hat{l}}:=\inf\{n\colon (X_n-X_0)\cdot\hat{l}\geq L\}$. Then the distribution of the random polymer is given by the Gibbs measure
\begin{eqnarray*}
dP_{L,\omega}^{\beta,\lambda}:=\frac{1}{Z_{L,\omega}^{\beta,\lambda}} e^{-\sum_{n=1}^{T_L^{\hat{l}}}(\lambda+\beta\omega(X_n))}dP
\end{eqnarray*}
where $Z_{L,\omega}^{\beta,\lambda}:=E[e^{-\sum_{n=1}^{T_L^{\hat{l}}}(\lambda+\beta\omega(X_n))} ]$, is the partition function, $\beta>0$ is the inverse temperature. The parameter $\lambda$ is strictly positive and adds an additional penalization to the paths, which take {\it very long time} to reach the hyperplane in direction $\hat{l}$, lying at distance $L$ from the origin. This has the effect that the path feels an additional drift towards direction $\hat{l}$, which justifies the term {\it semidirected}. We will make this point more precise later on. 

Semidirected polymers can be considered as a generalization of {\it directed polymers}, which are known to exhibit a very rich phenomenology. It is expected that the qualitative features of these phenomena should appear in the semidirected case, as well. A great difficulty establishing these features in the semidirected case is that most of the techniques used in the study of directed polymers are based on martingale arguments (with the notable exception of \cite{L}). The martingale formulation is inherent in the directed case, since the path does not visit the same site twice. In the semidirected case, though, the path can visit the same site many times and this introduces correlations which destroy the martingale structure. Therefore, one has to resort to other more quantitative methods of analysis. An attempt towards this direction was recently initiated in \cite{F1}, \cite{Zy}, \cite{IV}, \cite{KMZ} and a purpose of the present work is to continue building towards this direction. Before describing the goals of this work let us review some of the basic notions and current results. A more complete review on the subject appears in the recent article \cite{IVReview}.

A fundamental quantity is the point-to-hyperlane quenched Lyapounov norm
\begin{eqnarray}\label{quenched_Lyap}
\alpha_\lambda^*(\hat{l}) :=-\lim_{L\to\infty}\frac{1}{L}\log E\left[e^{-\sum_{n=1}^{T_L^{\hat{l}} }(\lambda+\beta \omega(X_n))}\right],
\end{eqnarray}
defined for any unit vector $\hat{l}\in \mathbb{Z}^{d}$. $\alpha_\lambda^*(\hat{l})$ is known to be independent of the realization of the disorder $\omega$, i.e. the limit exists $\mathbb{P}-a.s.$, and it can be extended so that to define a norm on $\mathbb{R}^{d}$ \cite{Z},\cite{Sz2}. This norm can be thought of as a measure of the {\it cost} that the random walk $(X_n)_{n\geq 1}$
has to pay in order to perform a long crossing among the potential $-(\lambda+\beta\omega(x))$, $x\in \mathbb{Z}^{d}$, or
 alternatively as the quenched free energy of the semidirected polymer in direction $\hat{l}$. 
 The point-to-hyperplane Lyapounov norms, as well as their dual point-to-point norms $\alpha_\lambda(x):=\sup_{\hat{l}\in \mathbb{R}^{d}} x\cdot\hat{l}/\alpha_\lambda^*(\hat{l})$, were first introduced by Sznitman as part of the program of studying the detailed large deviation properties of Brownian motion among Poissonian obstacles. We will not detail further on this very interesting aspect, but a complete amount of this work can be found in \cite{Sz2}, chapters 5 and 7. Let us point out that the results proved in the present paper could be translated in order to yield information on the above mentioned large deviations rate functions. We will not go down this route, though, since our main focus is to establish a phenomenology on semidirected polymers in analogy with directed polymers.
 
 Significant amount of information about the path properties of the semidirected polymer can be deduced from the study of the quenched Lyapounov norm and in particular from its comparison with the annealed Lyapounov norm. The latter is defined as follows
\begin{eqnarray}\label{annealed_Lyap}
\beta_\lambda^*(\hat{l}) :=-\lim_{L\to\infty}\frac{1}{L}\log \mathbb{E}E\left[e^{-\sum_{n=1}^{T_L^{\hat{l}} }(\lambda+\beta \omega(X_n))}\right]
\end{eqnarray}
Borrowing the terminology from directed polymers, we will say that {\it strong disorder} holds when the annealed and the quenched Lyapounov norms are different. Since it is always the case that $\alpha_\lambda^*(\hat{l})\geq \beta_\lambda^*(\hat{l}) $, strong disorder amounts to a strict inequality between the norms.

It was established in \cite{Zy} that for any $\hat{l}\in \mathbb{Z}^{d}$, when $d\geq 4$ and $\beta<\beta_0(\lambda)$, strong disorder fails, that is $\alpha_\lambda^*(\hat{l})= \beta_\lambda^*(\hat{l}) $, for every $\hat{l}\in \mathbb{Z}^{d}$. In the case that $\hat{l}$ is parallel to a vector of the standard orthonormal basis of $\mathbb{R}^d$ this result was also established in \cite{F1}. A stronger result can be deduced from the proof in \cite{F1}, \cite{Zy}; namely, that when $\hat{l}$ is parallel to a vector of the standard orthonormal basis of $\mathbb{R}^d$, the value $\beta_0$ is independent of $\lambda$ .The same is expected to hold for arbitrary $\hat{l}$. Recently the equality of the Lyapounov norms was strengthened in \cite{IV}, by establishing that, in the same regime of parameters and for $\hat{l}$ parallel to a vector of the standard orthonormal basis of $\mathbb{R}^d$, the limit
$ E[e^{-\sum_{n=1}^{T_L^{\hat{l}} }(\lambda+\beta \omega(X_n))}] / \, \mathbb{E} E[e^{-\sum_{n=1}^{T_L^{\hat{l}} }(\lambda+\beta \omega(X_n))}]$
exists $\mathbb{P}-a.s.$ and it is strictly positive. Furthermore, it was established that, in this regime, the transversal to $\hat{l}$,  
location of the end point $X(T_L^{\hat{l}})$ of the path, satisfies a central limit theorem in $\mathbb{P}$-probability, extending partially in this way the corresponding picture that is valid in directed polymers \cite{B}.

In this paper we will establish the complementary results. Namely, we will identify situations where strong disorder holds and we will further prove that at strong disorder the semidirected polymer exhibits localization phenomena. To be more precise let us state our results. To simplify things we will restrict ourselves to the situation where $\hat{l}=\hat{e}_1$, with $\hat{e}_1,\dots,\hat{e}_{d}$ the canonical basis of $\mathbb{Z}^{d}$. We will also simplify the notation by denoting $\alpha_\lambda^*:=\alpha_\lambda^*(\hat{e}_1)$ and $\beta_\lambda^*:=\beta_\lambda^*(\hat{e}_1)$. Our first result is that
\begin{theorem}\label{Thm 1}
Assume that the disorder $\omega$ is nonnegative, does not concentrate on a single point and $\mathbb{E}[\omega^2]<\infty$. 

A. For any $\lambda>0$, $\beta>0$ and $d=2,3$ we have that
$\alpha_{\lambda}^*>\beta_{\lambda}^*.$

B. The strict inequality between the annealed and quenched norms is also valid in any dimension, if $\beta$ is large enough and the disorder satisfies the additional assumptions that $\text{essinf} \,(\omega)=0$ and $\mathbb{P}(\omega=0)<p_d $, where $p_d$ is the critical probability for site percolation in $\mathbb{Z}^d$.
\end{theorem}
The first result identifies certain situations where the Lyapounov norms are different. The case $d=1$ is not included, since it can be easily deduced from the work of Sznitman \cite{Sz2}, pg. 233, that strong disorder holds in this case. The one dimensional case is particular since one
can make a more quantitative use of the ergodic theorem. Theorem \ref{Thm 1} also identifies situations where strong disorder is valid, due to the presence of {\it low temperature}, i.e. large $\beta$. Notice that the assumption $\text{essinf}(\omega)=0$ is just a normalization, as we could readjust the value of the parameter $\lambda$.  It is a very interesting, open problem to obtain a quantitative description in dimensions three and above of the phase transition between weak and strong disorder. In the case of directed polymers the existence of a critical value $\beta_c(d)$ separating the two phases has been established \cite{CY}, Lemma 3.3, but such a separation has not been established, yet, for semidirected polymers. Even more interesting would be to understand how this phase transition depends on the distribution of the disorder, as well as the dimension. This type of question is also widely open for directed polymers, although, in that setting  a non-quantitative characterization of $\beta_c(d)$, based on martingale arguments \cite{CSY}, exists.

Our second result is concerned with the distribution of the end point of the semidirected polymer when it reaches a hyperplane at distance $L$ from the origin. To this end let us define the measure
\begin{eqnarray}\label{p-t-h-measure}
\mu_{L,\omega}^{\beta,\lambda}(x):=
\frac{E\left[e^{-\sum_{n=1}^{T_L} (\lambda+\beta \omega(X_n) )} ;X(T_L) =x\right] }
{E\left[e^{-\sum_{n=1}^{T_L} (\lambda+\beta \omega(X_n) )}\right]},
\end{eqnarray}
where $T_L:=\inf\{n\colon (X_n-X_0)\cdot\hat{e}_1\geq L\}$. Then we have
\begin{theorem}\label{Thm 2}
If $\alpha_\lambda^*>\beta_\lambda^*$ then $\mathbb{P}-a.s.$ we have that 
$$\limsup_{L\to\infty}\sup_{x\colon x\cdot \hat{e}_1=L}\mu_{L,\omega}^{\beta,\lambda}(x) >0.$$
\end{theorem}

This result should be contrasted with the one in \cite{IV} about diffusive bahavior in the case of small $\beta$ and high dimension. Our result indicates that the mass of the distribution of the polymer, when this reaches certain hyperplanes, does not spread out as in the case of diffusive behavior. Instead it develops atoms, which means that there are areas (whose location is random) on the various hyperplanes, where the polymer concentrates with high probability. In other words, the polymer localizes. This type of localization is known to exist in directed polymers \cite{CSY}, \cite{V} and our result can be viewed as an extension to the semidirected case.

The organisation of the paper is as follows. In section \ref{notation} we introduce the necessary notation and recall a number of basic results upon which the analysis is based. 
Most of these appear in previous works and we try to sketch the proofs of a number of them. A number of new auxiliary results, is also included. In Section \ref{large beta} we prove part B of Theorem \ref{Thm 1}. In Section \ref{proof of Thm 1} we prove part A of Theorem \ref{Thm 1}. Here we use the method of estimating fractional moments in the way this was developed through the study of random pinning polymers \cite{GLT} and applied to directed polymers \cite{L}. The successful application of the fractional moment method in our case builds crucially on a certain renewal structure of the semidirected polymers. Finally in Section \ref{proof of Thm 2} we prove the localization property stated in Theorem \ref{Thm 2}. 

Let us make a note on notation. $C$ will denote some generic constant, whose values do not depend on any of the other parameters, e.g. $\lambda,\beta, d$, etc. and whose value may be different in different appearances. In the case of some important constants, whose value needs to be distinguished we will enumerate them, i.e. $C_1,C_2,$ etc. When we want to stress the dependence of the constant on some other parameter we will indicate this by a subscript, e.g. $C_\epsilon$. We will also frequently use the decomposition $x:=(x^{(1)},x^\perp)$, for an arbitrary point $x\in \mathbb{Z}^{d}$, where $x^{(1)}:=x\cdot \hat{e}_1\in \mathbb{Z}$ and $x^\perp\in \mathbb{Z}^{d-1}$. For a set $A$, we will denote by $\overline{A}$ its complement. Finally in order to lighten the notation we will refrain from using the symbol $[x]$ to denote the integer part of a parameter $x$ and instead we will be using the symbol $x$ having in mind that it means the {\it closest} lattice point to $x$. It is unlikely that this convention will lead to any confusion, but on the other hand it will make the notation much lighter.
 
\section{ Notation and preliminary results.}\label{notation}

Let us define the local time at $x\in \mathbb{Z}^{d}$, between times $M,N$ by $\ell_{M,N}(x)=\sum_{n=M}^{N-1} 1_{x}(X_n)$. Whenever $M=0$ we will simply denote it by $\ell_{N}(x)$. To simplify notation we will also drop the subscript $N$ when it is clear when is the terminal time within which we consider the local time.

We denote by $\mathcal{H}_L:=\{x\in \mathbb{Z}^{d}\colon x\cdot\hat{e}_1=L\}$. We also define the hitting time and the hitting point of the hyperplane at distance $L$ from the starting point of the walk (this being $\mathcal{H}_L$ if the starting point is the origin) as
\begin{eqnarray*}
&&T_L:=\inf\{ n: (X_n-X_0)\cdot \hat{e}_1 =L\},\\
&&\hat{X}_{L}:=X(T_{L}).
\end{eqnarray*} 
and the last hitting point of the hyperplane at distance $L$ from the starting point of the walk as
\begin{eqnarray*}
S_L:=\sup\{ n: (X_n-X_0)\cdot \hat{e}_1 =L\}.
\end{eqnarray*}

We already mentioned in the introduction that the parameter $\lambda$ introduces an effective drift  towards direction $\hat{e}_1$. 
We will make this point more explicit by the use of a Girsanov type argument. It will turn out that this formulation is more convenient and we will adapt it through out the rest of the paper. To be more precise let $P^\kappa_x$ the distribution of a random walk starting from $x\in \mathbb{Z}^{d}$
with transition  probabilities
\begin{equation}\label{walk_drift}
\pi_{\kappa}(x,y)=\left\{\begin{array}{ll}
\frac{e^{\kappa (y-x)\.\hat{e}_1}}{2(\cosh(\kappa)+d-1)} &, \text{if} \quad |x-y|=1,\\
0& , \text{if} \quad |x-y|\neq 1.
\end{array}\right.
\end{equation}
The parameter $\kappa:=\kappa(\lambda,\beta)$ is chosen so that it satisfies the equation
\begin{equation}\label{hl}
\log\left(\frac{\cosh(\kappa)+d-1}{d}\right)=\lambda +\beta\mathbb{E}[\omega]. 
\end{equation}
In other words $P^\kappa_x$ is a random walk with a drift towards the $\hat{e}_1$ direction. As usual, we will not include the subscript $x$ when this coincides with the origin. It will also be convenient to center the disorder $\omega$. To this end we write $\overline\omega_x:=\omega_x -\mathbb{E}[\omega_x]$, for every $x\in \mathbb{Z}^{d}$ and we have
\begin{eqnarray*}
E\left[ e^{-\sum_{n=1}^{T_L}(\lambda+\beta\omega(X_n))}\right]&=&
  E\left[ e^{-\sum_{n=1}^{T_L}(\lambda+ \beta\mathbb{E}[\omega]+\beta\overline\omega(X_n))}\right]\\
  &=&E^\kappa\left[ \frac{dP}{dP^\kappa}\Big|_{\mathcal{F}_{T_L}}e^{-\sum_{n=1}^{T_L}(\lambda+ \beta\mathbb{E}[\omega]+\beta\overline\omega(X_n))}\right]\\
  &=&E^\kappa\left[ e^{-\kappa\sum_{n=1}^{T_L}(X_n-X_{n-1})\cdot\hat{e}_1 -T_L\log(d/(\cosh{\kappa}+d-1))} e^{-\sum_{n=1}^{T_L}(\lambda+ \beta\mathbb{E}[\omega]+\beta\overline\omega(X_n))}\right].
\end{eqnarray*}
 Here $ \mathcal{F}_{T_L}$ is the $\sigma-$algebra generated by the first $T_L$ steps of the random walk. The choice of $\kappa$ in \eqref{hl} and the fact that 
  $\kappa\sum_{n=1}^{T_L}(X_n-X_{n-1})\cdot\hat{e}_1=\kappa (X_{T_L}-X_0)\cdot\hat{e}_1=\kappa L$ gives that the above is equal to
  \begin{eqnarray}\label{weighted partition}
  e^{-\kappa L} E^\kappa\left[  e^{ -\sum_{n=1}^{T_L} \beta\overline\omega(X_n)}  \right].
  \end{eqnarray}
From this it is evident that
\begin{eqnarray*}
\alpha_\lambda^*=\kappa -\lim_{L\to\infty}\frac{1}{L}\log E^\kappa\left[  e^{ -\sum_{n=1}^{T_L} \beta\overline\omega(X_n)}  \right],
\end{eqnarray*}
and  
\begin{eqnarray*}
\beta_\lambda^*=\kappa -\lim_{L\to\infty}\frac{1}{L}\log \mathbb{E} E^\kappa\left[  e^{ -\sum_{n=1}^{T_L} \beta\overline\omega(X_n)}  \right].
\end{eqnarray*}
Let us also denote the log-moment generating function of $\overline\omega$ by 
\begin{eqnarray}\label{log-mom}
\overline\phi(t):=-\log \mathbb{E}\left[e^{-t\overline\omega}\right],
\end{eqnarray}
and the annealed potential 
\begin{eqnarray}\label{annealed potential}
\overline\Phi_\beta(M,N):=-\log \mathbb{E}\left[ e^{-\beta\sum_x\overline\omega(x)\ell_{M,N}(x)} \right]
=\sum_x \overline\phi(\beta\ell_{M,N}(x)).
\end{eqnarray}
Again, when $M=0$ we will simply denote this by $\overline\Phi_\beta(N)$.
The next proposition collects some properties of the function $\overline\Phi_\beta$, which are useful and easy to verify. Here, we will only give a sketch of the proof.
\begin{proposition}\label{properties}
(i) For $M,N$ integers we have that 
\begin{equation*}
\overline\Phi_\beta(M+N)\leq\overline\Phi_\beta(N)+\overline\Phi_\beta(N,N+M).
\end{equation*}
(ii) Let $N_1<N_2<N$, then 
\begin{equation*}
\overline\Phi_\beta(N)\geq\overline\Phi_\beta([0,N_1]\cup [N_2,N])-\beta\mathbb{E}[\omega]\,(N_2-N_1).
\end{equation*}
(iii) If $ (X_n)_{0\leq n< N_1}\cap (X_n)_{N_1 \leq  n\leq N_1+N_2}=\emptyset$, then
\begin{equation*}
\overline\Phi_\beta(N_1+N_2)=\overline\Phi_\beta(N_1)+\overline\Phi_\beta(N_1,N_1+N_2).
\end{equation*} 
\end{proposition}
The notation used on the right hand side of $(ii)$ means that in the evaluation of $\overline\Phi_\beta([0,N_1]\cup [N_2,N])$ we consider the local time
$\ell_{[0,N_1]\cup [N_2,N]}:=\ell_{N_1}+\ell_{N_2,N}$. The proof of $(ii)$ makes use of the monotonicity $\ell_{N}\geq \ell_{[0,N_1]\cup [N_2,N]}$ and the fact that the potential $\beta\overline\omega$ is bounded below by $-\beta\mathbb{E}[\omega]$. The proof
 of $(iii)$ uses the independence of the potentials
visited by the two parts of the walk. Finally, the proof of $(i)$ makes an easy use of H\"older's inequality. Alternatively one can deduce it via the Harris-FKG inequality of positive association.

The following definitions set the grounds upon which the analysis of semidirected polymers is based. The notions of {\it break points} and {\it irreducible bridges}, presented below, are in the core of the renewal structure upon which the parallelisms with directed polymers are based. Formally speaking, a path going from the origin to a hyperplane, will have points in its trajectory with the property that, once the path reaches them, it does not backtrack in the future behind their level. Therefore, the trajectory can be decomposed into a sequence of non intersecting {\it cylinders}. What is important is that the range of the path within these cylinders as well as the potential encountered by the corresponding parts of the path are independent with each other.

\begin{definition}\label{Bridge}
(i) Consider the walk $(\,X_n\,)_{M\leq n\leq N}$. 
We will say that the walk forms a bridge of span $L$,
 and  denote it by $\Br(M,N;L)$, if
\begin{equation*}
X_M\cdot\hat{e}_1\leq X_n\cdot\hat{e}_1<X_N\cdot\hat{e}_1,
\end{equation*}
for $M\leq n<N$, and $(X_N-X_M)\cdot\hat{e}_1=L$. When $M=0$, we will
write $\Br(N;L)$ instead.

(ii) Let us denote 
\begin{eqnarray*}
&&\overline{B}_{x,\omega}(L):=E_x^\kappa\left[ e^{-\sum_{n=0}^{T_L-1}\beta \overline\omega(X_n)  }; \text{Br}(T_L,L)\right]=
\sum_{N=1}^\infty E_x^\kappa\left[ e^{-\sum_{n=0}^{N-1}\beta \overline\omega(X_n)  }; \text{Br}(N,L)\right],
\end{eqnarray*}
and
\begin{eqnarray*}
\overline{B}_x(L)=E_x^\kappa\left[ e^{-\overline\Phi_\beta(T_L)}\,;\,\Br(T_L;L)\right]
=\sum_{N=1}^\infty E_x^\kappa\left[ e^{-\overline\Phi_\beta(N)}\,;\,\Br(N;L)\right].
\end{eqnarray*}
If $A$ is an event on the random walk we denote
\begin{eqnarray*}
&&\overline{B}_{x,\omega}(L;A):=E_x^\kappa\left[ e^{-\sum_{n=0}^{T_L-1}\beta \overline\omega(X_n)  }; \text{Br}(T_L;L)\cap A\right],
\end{eqnarray*}
and
\begin{eqnarray*}
\overline{B}_x(L;A)=E_x^\kappa\left[ e^{-\overline\Phi_\beta(T_L)}\,;\,\Br(T_L;L)\cap A\right].
\end{eqnarray*}
\end{definition}
\begin{definition}\label{break_point}
Consider the random walk $(X_n)_{M\leq n\leq N}$. We will say that the random walk has a break point at
level $L$, if there exists an $n$, with $M<n<N$ such that $X_n\cdot\hat{e}_1=L$ and 
\begin{equation*}
X_{n_1}\cdot\hat{e}_1<X_n\cdot\hat{e}_1\leq X_{n_2}\cdot\hat{e}_1,
\end{equation*}
for $M\leq n_1<n\leq n_2\leq N$.
\end{definition}
\begin{definition}\label{Irreducible_bridge}
(i)
Consider the random walk $(\,X_n\,)_{M\leq n\leq N}$.
We will say that the random walk forms an irreducible bridge
of span $L$, and we denote it by $\Ir(M,N;L)$,
if it forms a bridge of span $L$ with no break points.
When $M=0$ we will write $\Ir(N;L)$ instead.

(ii) Let us denote 
\begin{eqnarray*}
&&\overline{I}_{x,\omega}(L):=E_x^\kappa\left[ e^{-\sum_{n=0}^{T_L-1}\beta \overline\omega(X_n)  }; \text{Ir}(T_L;L)\right]=
\sum_{N=1}^\infty E_x^\kappa\left[ e^{-\sum_{n=0}^{N-1}\beta \overline\omega(X_n)  }; \text{Ir}(N;L)\right],
\end{eqnarray*}
and
\begin{eqnarray}
\overline{I}_x(L)
=E_x^\kappa\left[ e^{-\overline\Phi_\beta(T_L)}\,;\,\Ir(T_L;L)\right]
=\sum_{N=1}^\infty E_x^\kappa\left[ e^{-\overline\Phi_\beta(N)}\,;\,\Ir(N;L)\right].
\end{eqnarray}
If $A$ is an event on the random walks we denote
\begin{eqnarray*}
&&\overline{I}_{x,\omega}(L;A):=E_x^\kappa\left[ e^{-\sum_{n=0}^{T_L-1}\beta \overline\omega(X_n)  }; \text{Ir}(T_L;L)\cap A\right],
\end{eqnarray*}
and
\begin{eqnarray*}
\overline{I}_x(L;A)=E_x^\kappa\left[ e^{-\overline\Phi_\beta(T_L)}\,;\,\Ir(T_L;L)\cap A\right].
\end{eqnarray*}
\end{definition}
As usual we will refrain from including the subscript $x$ in the above definitions, when this coincides with the origin.
\begin{definition}
 Let us define the quenched and annealed mass for bridges, respectively, by
 \begin{eqnarray}\label{quenched_mass}
\overline{m}_B^q&:=&\lim_{L\to\infty}-\frac{1}{L}\log
 E^\kappa\left[ e^{-\sum_{n=0}^{T_L-1} \beta\,\overline\omega(X_n)}\,;\Br(T_L;L)\,\right]\nonumber\\
&=&\lim_{L\to\infty}-\frac{1}{L}\log \overline{B}_\omega(L).
\end{eqnarray}
and
\begin{eqnarray}\label{annealed_mass}
\overline{m}_B^a&:=&\lim_{L\to\infty}-\frac{1}{L}\log
\mathbb{E} E^\kappa\left[ e^{-\sum_{n=0}^{T_L-1} \beta\,\overline\omega(X_n)}\,;\Br(T_L;L)\,\right]\nonumber\\
&=&\lim_{L\to\infty}-\frac{1}{L}\log \overline{B}(L).
\end{eqnarray}
\end{definition}
The existence of these limits follows standard subadditive arguments. The bridge masses $\overline{m}_B^a,\overline{m}_B^q$ depend on the parameters $\lambda, \beta$, but for simplicity we will not include this dependence in the notation.  
It is easy to see that $\alpha_\lambda^*=\kappa+\overline{m}_B^q$ and that $\beta_\lambda^*=\kappa+\overline{m}_B^a$. Therefore Theorems  \ref{Thm 1} and \ref{Thm 2} can be recast in terms of the quenched and annealed masses. In fact our main goal will be to prove that $\overline{m}_B^q>\overline{m}_B^a$  under the conditions of Theorem \ref{Thm 1} and then that the situation of strict inequality between the masses implies that the conclusion of Theorem \ref{Thm 2} holds.
 
The following proposition will be useful, since it gives a uniform bound on the decay of $\overline{B}(L)$.
\begin{proposition}\label{subadditivity}
There exists a constant $\mu_0<1$, such that for every $L$, 
$$\mu_0e^{-\overline{m}_B^a L}\leq \overline{B}(L)\leq e^{-\overline{m}_B^a L}.$$
\end{proposition}
 The proof of the right hand inequality is based on the basic supermultiplicative property of bridges, that is $\overline{B}(L_1+L_2)\geq \overline{B}(L_1)\overline{B}(L_2)$, for any $L_1,L_2$. This can be deduced from the inclusion $\text{Br}(T_{L_1+L_2};L_1+L_2 ) \supset \text{Br}(T_{L_1};L_1)\cap \text{Br}(T_{L_1},T_{L_1,L_1+L_2};L_2)$ and property $(iii)$ of Proposition \ref{properties}. The left hand side inequality is based on the reverse multiplicative property of the annealed potential, that is $\overline{B}(L_1+L_2)\leq \overline{B}(L_1)\overline{B}(L_2)\\ \times\sum_{n=1}^\infty e^{\beta \mathbb{E}[\omega]n} P^\kappa (X_n^{(1)}=0)$. This is easily deduced by bounding from below the potential $\beta\overline{\omega}$, encountered by the part of the path between the first and the last time that it lies on level $\mathcal{H}_{L_1}$, by $-\beta \mathbb{E}[\omega]$. An easy computation shows that $\sum_{n=1}^\infty e^{\beta \mathbb{E}[\omega]n} P^\kappa (X_n^{(1)}=0)=\sum_{n=1}^\infty e^{-\lambda n}P(X_n^{(1)}=0)<\infty$. $\mu_0$ is then chosen to be $\left(\sum_{n=1}^\infty e^{-\lambda n}P(X_n^{(1)}=0)\right)^{-1}$.

We also define
\begin{eqnarray}\label{normalized_B}
\hat{B}_\omega(L):=e^{\overline{m}_B^a L}\overline{B}_\omega(L).
\end{eqnarray}
and
\begin{eqnarray}\label{normalized_Bann}
\hat{B}(L):=e^{\overline{m}_B^a L}\overline{B}(L).
\end{eqnarray}
Central tool in the study of semidirected polymers is the renewal structure, which governs the annealed and the quenched irreducible bridges. This is summarised in the relation 
\begin{eqnarray}\label{renewal}
\overline{B}(L)=\sum_{k=1}^N\overline{I}(k)\overline{B}(L-k),
\end{eqnarray} 
which can be obtained by decomposing the bridge $\overline{B}(T_L;L)$ according to when the first break point occurs. 
A number of very useful properties can be deduced from the relation \eqref{renewal}. The most fundamental one is that 
\begin{eqnarray}\label{probability}
\sum_{L=1}^\infty e^{\overline{m}_B^aL} \overline{I}(L)=1.
\end{eqnarray}
The proof of this statement follows a generating functions calculation in the frame of standard renewal theory together with the lower estimate of Proposition \ref{subadditivity}. The details of the proof (with a little different notation) can be found either in \cite{Zy}, Proposition 4.2, or \cite{F1}, Lemma 2.15. It follows that  $(\hat{I}(L))_{L=1,2,\dots}:=(e^{\overline{m}_B^aL} \overline{I}(L ))_{L=1,2,\dots}$ is a probability distribution and moreover, it has exponential moments. In particular,we have that
\begin{proposition}\label{massgap}
There exists a $\rho=\rho(\lambda)>0$ such that, for any $\beta>0$
\begin{eqnarray*}
\sum_{L=1}^\infty e^{(\rho+\overline{m}_B^a)L}\overline{I}(L)<\infty.
\end{eqnarray*}
\end{proposition}
Such an estimate is known as {\it mass gap estimate}. In the context of self-avoiding walks it was first proven in \cite{CC}. It was later adapted to the context of random walks in random potentials in \cite{T},\cite{F1}. Such type of estimate in the context of Lyapounov norms, independent of the direction and for small $\beta$ was established in \cite{Z}  (in this same work a separate proof, valid for all $\beta$, along coordinate directions, that was meant to simplify the existing ones, appears to be flawed). Finally mass gap estimates that also apply on different contexts, such as Ising models appear in \cite{CIV} and \cite{IV2}.  
A proof of Proposition \ref{massgap} most relevant to our setting (with a little different notation) can be found in \cite{F1}, Theorem 2.18.

We will denote the mean of this probability distribution by $\mu$, i.e.
\begin{eqnarray}\label{mu}
\mu:=\sum_{L=1}^\infty L e^{\overline{m}_B^a L}\overline{I}(L)<\infty.
\end{eqnarray}
 Using standard renewal theory arguments one can also easily deduce that $\hat{B}(L)\to\mu^{-1}$, as $L$ tends to infinity, refining in a sense Proposition \ref{subadditivity}.
\vskip 2mm
The importance of the above considerations is that they lead to a Markovian structure of the triplet $(X_{\tau_i}, \mathcal{L}_i,\tau_i)$, where $\tau_i$ denotes the time when the $i^{th}$ break point occurs, $\mathcal{L}_i$ the span of the $i^{th}$ irreducible bridge and $X_{\tau_i}$ the position of the path at the break point. This Markovian structure is central in our considerations and is described by the following Markov measure.
\begin{definition} The measure $P^\beta$ denotes the distribution of the Markov process $(X_{\tau_i}, \mathcal{L}_i,\tau_i)$ with transition probabilities given by
\begin{eqnarray*}
&&p^\beta(y_{i+1},L_{i+1},n_{i+1};y_i,L_i,n_i):=e^{\overline{m}_B^aL_{i+1}}\\ 
&&\qquad\times\,\, E^\kappa_{y_i}\left[ e^{-\overline\Phi_\beta(n_{i+1}-n_i)};\Ir(n_{i+1}-n_i,L_{i+1}), \,X_{n_{i+1}-n_i}=y_i\right].
\end{eqnarray*}
\end{definition}
It follows from Proposition \ref{massgap} that $E^\beta[e^{\rho \mathcal{L}_1}]<\infty$. Some further elaboration on this relation leads to the following proposition
\begin{proposition}\label{exponential_moments}
There exists $\rho_1=\rho_1(\beta,\lambda)>0$ such that, for every $\beta>0$
\begin{eqnarray*}
P^\beta(\tau_1>u)\leq e^{-\rho_1u}.
\end{eqnarray*}
It moreover follows that $E^\beta[e^{\frac{\rho_1}{2}\sup_{n<\tau_1} |X_n|}]<\infty$.
\end{proposition}
\begin{proof}
We have that
\begin{eqnarray}\label{tau 1 exponential}
P^\beta(\tau_1>u)\leq P^\beta(\mathcal{L}_1>hu)+ P^\beta(\tau_1>u;\,\mathcal{L}_1\leq hu),
\end{eqnarray}
for some $h$ small enough that will be chosen below. The first term in \eqref{tau 1 exponential} is estimated by
\begin{eqnarray*}
P^\beta(\mathcal{L}_1>hu)\leq
e^{-\rho h u} E^\beta[e^{\rho\mathcal{L}_1}]\leq Ce^{-\rho h u},
\end{eqnarray*}
where the last inequality is thanks to Proposition \ref{massgap}. Regarding the second term in \eqref{tau 1 exponential}, this is estimated as follows
\begin{eqnarray*}
 P^\beta(\tau_1>u;\,\mathcal{L}_1\leq hu) &=& 
    \sum_{N>u}\sum_{L\leq hu} e^{\overline{m}_B^aL} E^\kappa\left[ e^{-\overline\Phi_\beta(N)};\,\Ir(N;L)\,\right]\\
    &=&  \sum_{N>u}\sum_{L\leq hu} e^{(\overline{m}_B^a+\kappa)L} e^{-\kappa L}E^\kappa\left[ e^{-\overline\Phi_\beta(N)};\,\Ir(N;L)\,\right]\\
    &=& \sum_{N>u}\sum_{L\leq hu} e^{(\overline{m}_B^a+\kappa)L} \mathbb{E}E\left[ e^{-\lambda N-\sum_{n=1}^N\beta\omega(X_N)};\,\Ir(N;L)\,\right],
\end{eqnarray*}
where the last follows as in \eqref{weighted partition}, with $\kappa$ being as in \eqref{hl}. Continuing on the above we have that
\begin{eqnarray*}
P^\beta(\tau_1>u;\,\mathcal{L}_1\leq hu)&\leq& e^{(\overline{m}_B^a+\kappa)hu}e^{-\frac{\lambda}{2}u} 
\sum_{N>u}\sum_{L\leq hu} \mathbb{E} E\left[ e^{-\frac{\lambda}{2} N-\sum_{n=1}^N\beta\omega(X_N)};\,\Ir(N;L)\,\right]\\
&\leq& e^{-\frac{\lambda}{4}u},
\end{eqnarray*}
with the last inequality valid if $h\leq4^{-1}(\overline{m}_B^a+\kappa)^{-1}\lambda$. Combining the two estimates we have that
\begin{eqnarray*}
P^\beta(\tau_1>u)\leq Ce^{-\rho h u}+e^{-\frac{\lambda}{4}u}.
\end{eqnarray*}
From this, the proposition follows with $\rho_1:=\rho_1(\beta,\lambda)=\min(\lambda/4,\rho(\overline{m}_B^a+\kappa)^{-1}\lambda/4)$.
\end{proof}
The following local limit theorem, proven in \cite{BS}, Theorem 5.1, will be useful towards Proposition \ref{local CLT}, below.
\begin{theorem}\label{BS}(\cite{BS})
Consider a distribution $p(\cdot)$ on $\mathbb{Z}^d, d\geq 1$ with covariance matrix $\Sigma_p$ and mean $\mu_p:=\sum_{x\in \mathbb{Z}^d} x p(x)$ satisfying
\begin{eqnarray}\label{g1}
\sum_{x\in \mathbb{Z}^d} p(x) e^{\gamma_0 |x|} \leq \gamma_1,
\end{eqnarray}
\begin{eqnarray}\label{g2}
 \Sigma_p\geq \gamma_2 I_d,
\end{eqnarray}
for certain constants $\gamma_0,\gamma_1,\gamma_2$.
Denote by $p_n(\cdot)$ its $n^{th}$ convolution. Then there exists an $\epsilon(\gamma_0, \gamma_1,\gamma_2)$ such that for $\epsilon<\epsilon(\gamma_0, \gamma_1,\gamma_2)$ there are positive constants $\tilde\delta:=\tilde\delta_{\gamma_0, \gamma_1,\gamma_2}, \tilde\delta_\epsilon:=\tilde\delta_\epsilon(\gamma_0, \gamma_1,\gamma_2)$ and $C=C(\gamma_0,\gamma_1,\gamma_2)$, such that
\begin{eqnarray}\label{local CLT bound}
p_n(x)\leq  \varphi_n^C(x) 1_{|x-n\mu_p|<n\epsilon} +Ce^{-\tilde\delta_\epsilon|x-n\mu_p|} 1_{|x-n\mu_p|\geq n\epsilon},
\end{eqnarray}
where
\begin{eqnarray*}
 \varphi_n^C(x):=\frac{C}{n^{d/2}} e^{-\frac{\tilde\delta|x-n\mu_p|^2}{2n}}.
\end{eqnarray*}
\end{theorem}
We close this section with the following local limit type annealed estimate that will be useful in several occasions in our estimates towards the inequality of the norms in dimensions two and three. 
\begin{proposition}\label{local CLT}
For $x^\perp\in \mathbb{Z}^{d-1}$ and $L>0$ integer, there is a constant $C$ such that
\begin{eqnarray*}
\hat{B}(L;\hat{X}_L=x)\leq \frac{C}{L^{(d-1)/2}}e^{-\frac{C|x^\perp|^2}{2L}}.
\end{eqnarray*} 
\end{proposition}
\begin{proof}
We use Proposition \ref{BS} with the distribution $p(\cdot)$ to be defined as
\begin{eqnarray}\label{p beta}
p^\beta(L,x^\perp):=P^\beta(\mathcal{L}_1=L;\hat{X}^\perp(\tau_1)=x^\perp),
\end{eqnarray}
for $L>0$ integer and $x^\perp\in \mathbb{Z}^{d-1}$.  Condition \eqref{g1} is satisfied by Proposition \ref
{exponential_moments}, $\mu_p$ is given by relation \eqref{mu} and, finally, \eqref{g2} is satisfied by the apparent non degeneracy of the distribution $P^\beta$. 
We then have 
\begin{eqnarray*}
\hat{B}(L;\hat{X}_L=x)&=&\sum_n p^\beta_n(L,x^\perp),
\end{eqnarray*}
and the result easily follows by inserting \eqref{local CLT bound} in the above summation.
\end{proof} 

\section{Inequality of Lyapounov norms for large $\beta$.}\label{large beta}

 In this section we prove part B of Theorem \ref{Thm 1}. The argument is along the lines of first passage percolation \cite{Kesten2}, similar to the case of Poissonian obstacles \cite{Szn1}, Theorem 1.4.  We consider the case when the distribution $\mathbb{P}$ satisfies $\text{essinf} \,(\omega)=0$ and $\mathbb{P}(\omega=0)<p_d$, with $p_d$ the critical probability for site percolation in $\mathbb{Z}^d$. Let $\phi(\beta):=-\log\mathbb{E}[\exp(-\beta\omega)]$.
We first have the following upper bound on $\beta^*_\lambda$.
\begin{eqnarray*}
\beta^*_\lambda&=&-\lim_{L\to\infty}\frac{1}{L}\log \mathbb{E}E\left[e^{-\sum_{n=1}^{T_L}(\lambda+\beta \omega(X_n))}\right]\\
&\leq& -\lim_{L\to\infty}\frac{1}{L}\log E\left[e^{-(\lambda +\phi(\beta))T_L }\right]\\
&\leq& C(\lambda+\phi(\beta)),
\end{eqnarray*}
where for the first inequality we used Fubini and the fact that $\mathbb{E}[\exp(-\beta \ell_{T_L}(x)\omega)]\geq (\mathbb{E}[\exp(-\beta\omega)])^ {\ell_{T_L}(x)}$. The second inequality is a routine to establish. Notice that since $\text{essinf}(\omega)=0$, we have that $\lim_{\beta\to\infty}\beta^{-1}\phi(\beta)=0$ and therefore it follows that for $\beta$ large, $\beta^*_\lambda=o(\beta)$.

We will now obtain a lower bound on the quenched Lyapounov norm. Consider $\omega^*_{d}$ the value such that 
$\mathbb{P}(\omega<\omega^*_{d})< p_{d}$, which is the critical probability of percolation in $d$ dimensions. By our assumption there exist such $\omega^*_d$, which is strictly positive.
Then by a first passage percolation argument (see Theorem 2.3 in \cite{Kesten2}, or Proposition 2.2 in \cite{Szn3})
 we have that for every $N>0$ there are constants $C_6,C_7$ such that
\begin{eqnarray}\label{fpp}
\mathbb{P}(\inf_{\mathcal{P}_N}\#\{n\leq N \colon \omega_{X_n}>\omega^*_{d}\} \leq C_6 N)\leq e^{-C_7 N},
\end{eqnarray}
where $\mathcal{P}_N$ is the set of all self avoiding paths $\{X_1,\dots,X_N\}$ of length N. Borel-Cantelli then implies that  for all large enough $N$ we have that  $\inf_{\mathcal{P}_N}\#\{n\leq N \colon \omega_{X_n}>\omega^*_{d}\} >C_6 N$. To use this in the estimate of the quenched Lyapounov norm, we notice that any path that starts at the origin makes at least $L$ steps before it reaches the hyperplane $\mathcal{H}_L$. We then have
\begin{eqnarray*}
\alpha_\lambda^*\geq -\lim\frac{1}{L}\log E\left[ e^{-\sum_{n=1}^{T_L} (\lambda+\beta\omega({X_n}))1_{\omega(X_n)>\omega^*_d}}\right]
\geq C_6(\lambda+\beta \omega_{d}^* ).
\end{eqnarray*}
Comparing this with the fact that $\beta^*_\lambda=o(\beta)$ for $\beta$ large, that we obtained above,  we arrive at the inequality of the Lyapounov norms, when $\beta$ is large.

\section{Inequality of Lyapounov norms in $d=2,3$ }\label{proof of Thm 1}
In this section we prove the first part of Theorem \ref{Thm 1}. The parameters $\lambda>0, \beta>0$ are fixed. To show that the annealed and the quenched Lyapounov norms are different  it is enough to show that
\begin{eqnarray}\label{fractional start}
\lim_{N\to \infty}\frac{1}{NL}\mathbb{E}\log \hat{B}_\omega(NL)<0,
\end{eqnarray}
recall that $\hat{B}_\omega(L):=e^{\overline{m}_B^aL}\overline{B}_\omega(L)$. This is evident, since the  left hand side of \eqref{fractional start} is equal to $\overline{m}_B^a-\overline{m}_B^q=\beta^*_\lambda-\alpha^*_\lambda$. To establish \eqref{fractional start}, we use the fractional moment method, which was developed in \cite{GLT}, \cite{L}. The starting point is to trivially write the left hand side of \eqref{fractional start} as 
$(\gamma NL)^{-1}\mathbb{E}\log \hat{B}^\gamma_\omega(NL)$, for a fixed $\gamma\in (0,1)$. Then using Jensen's inequality, it suffices to show that
\begin{eqnarray*}
\limsup_{N\to \infty}\frac{1}{\gamma NL}\log \mathbb{E}\hat{B}^\gamma_\omega(NL)<0.
\end{eqnarray*}
 The reason of considering a system of length $NL$ is that the fractional moment estimates are based on a coarse graining. The scale $L$ plays the role of a correlation length and its careful choice will be important. For the coarse graining we need to introduce the following skeletons:

Let $\mathcal{V}=\{0=v_0,v_1,\dots,v_N\} \subset \mathbb{Z}^{d}$, such that $v_i^{(1)}=i$, for $i=0,1,\dots N$. We define the skeletons 
\begin{eqnarray*}
\mathcal{I}^{C_1}_{\mathcal{V}}&:=&\bigcup_{i=0}^{N}\{ x\in \mathbb{Z}^{d}\colon x^{(1)}=v_{i}^{(1)}L, \,\,0\leq x^{(j)}-v^{(j)}_{i}C_1\sqrt{L}<C_1\sqrt{L},\,j=2,\dots,d\}\\
&:=&\cup_{i=0}^NI_{v_i}^{C_1},\\
\mathcal{J}_{\mathcal{V}}&:=&\bigcup_{i=0}^{N-1}\{x\in \mathbb{Z}^{d}\colon 0\leq x^{(1)}-v^{(1)}_{i}L<L, |x^{(j)}-v^{(j)}_{i}C_1\sqrt{L}|<C_3\sqrt{L} ,\,j=2,\dots,d\}\\
&:=& \cup_{i=0}^{N-1}J_{v_i}.
\end{eqnarray*}
We also define
\begin{eqnarray*}
I_0^{C_2}:=\{x\in \mathbb{Z}^d\colon x^{(1)}=0, |x^\perp|<C_2\sqrt{L}\}.
\end{eqnarray*}
The constants satisfy the relation $C_1<<C_2<<C_3$.
We now proceed as follows 
\begin{eqnarray}\label{fractional}
\frac{1}{\gamma}\frac{1}{NL}\log \mathbb{E}\left[\hat{B}_\omega(NL)^\gamma\right]   
&=&
\frac{1}{\gamma}\frac{1}{NL}\log \mathbb{E}\left[ \left(\sum_{\mathcal{V}} \hat{B}_\omega(NL;\cap_{i=1}^N \hat{X}_{iL}\in I^{C_1}_{v_i})\right)^\gamma\right]    \nonumber\\
&\leq&
\frac{1}{\gamma}\frac{1}{NL}\log \sum_{\mathcal{V}} \mathbb{E}\left[ \left(\hat{B}_\omega(NL;\cap_{i=1}^N \hat{X}_{iL}\in I^{C_1}_{v_i})\right)^\gamma\right].\qquad
\end{eqnarray}
Using H\"older's inequality we have the bound
\begin{eqnarray}\label{Holder}
&&\mathbb{E}\left[ \left(\hat{B}_\omega(NL;\cap_{i=1}^N \hat{X}_{iL}\in I^{C_1}_{v_i}\right)^\gamma\right] \nonumber\\
&&\qquad \qquad\leq
 \left(\mathbb{E}\left[g^{(d)}_{\mathcal{V}}(\omega)^{-\frac{\gamma}{1-\gamma}}\right]\right)^{1-\gamma}
 \left(\mathbb{E}\left[ g^{(d)}_{\mathcal{V}}(\omega) \,\hat{B}_\omega(NL;\cap_{i=1}^N \hat{X}_{iL}\in I^{C_1}_{v_i} )\right]\right)^\gamma,
 \end{eqnarray} 
where we define
\begin{eqnarray*}
g^{(d)}_{\mathcal{V}}(\omega)&:=&\prod_{i=0}^{N-1}g^{(d)}_{J_{v_i}}(\omega):=\exp\left(\sum_{i=0}^{N-1}F\left(G^{(d)}_{J_{v_i}}(\omega)\right)\right),\\
F(x)&:=&-K_1 1_{x>e^{K_2}},
\end{eqnarray*}
the constants $K_1,K_2$ will be chosen later on to be large enough. The index $d$ corresponds to the dimensions $2,3$, since we will require a different choice of the functions $G^{(d)}_{J_{v_i}}$ for each dimension. In particular, we will have the choices:
\begin{eqnarray}\label{d=2}
G^{(2)}_{J_{v_i}}(\omega)=\delta_L\sum_{x\in J_{v_i}}\overline\omega_x,
\end{eqnarray}
with $\delta_L:=-L^{3/4}$ and
\begin{eqnarray}\label{d=3}
G^{(3)}_{J_{v_i}}(\omega)=\sum_{y,z\in J_{v_i}}V_{y,z}\overline\omega_y\overline\omega_z,
\end{eqnarray}
with
\begin{eqnarray}\label{V}
V_{y,z}=\frac{1}{L(\log L)^{1/2}}\frac{1_{|y^\perp-z^\perp|<C_4\sqrt{|y^{(1)}-z^{(1)}|}}}{|y^{(1)}-z^{(1)}|+1}\,1_{y\neq z},
\end{eqnarray}
with the constant $C_4$ to be chosen large enough. The notation $\overline\omega_x:=\omega_x-\mathbb{E}[\omega_x]$, $x\in\mathbb{Z}^d$, will be used through out.
For shorthand we will be using the notation
\begin{eqnarray*}
d\mathbb{P}_{\mathcal{V}}&:=&g^{(d)}_{\mathcal{V}}(\omega) d\mathbb{P},\\
d\mathbb{P}_{J_{v_i}}&:=&g^{(d)}_{J_{v_i}}(\omega) d\mathbb{P},\qquad \text{for $i=1,2,\dots, N-1$},
\end{eqnarray*}
to denote the related measures. Notice that we have dropped the index $d$ from the notation of the $\mathbb{P}_{\mathcal{V}}$ and $\mathbb{P}_{J_{v_i}}$, in order to keep the notation light, since no confusion is likely to occur. 

Our goal will now be to use \eqref{Holder} into \eqref{fractional} in order to show that
\begin{eqnarray}\label{difference masses}
\overline{m}_B^a-\overline{m}_B^q\leq  \frac{1}{L}\limsup_{N\to\infty}\frac{1}{\gamma N}\log\sum_{\mathcal{V}}
     \mathbb{E}\left[ \left(\hat{B}_\omega(NL;\cap_{i=1}^N \hat{X}_{iL}\in I^{C_1}_{v_i})\right)^\gamma\right] <0,
\end{eqnarray}
when $L$ is chosen appropriately and large enough and so are the constants $K_1,K_2,C_1,C_2$, $C_3,C_4$.
To achieve this we will need a number of estimates. We start with the following estimate on the term involving the Radon-Nikodym derivative.
\begin{proposition}\label{Radon Nikodym}
For $K_2$ large enough, depending on $\mathbb{E}[\omega^2], K_1,C_3,C_4$ and for the above choices of the parameters $\delta_L$ and
$V_{y,z}$ we have that
\begin{eqnarray*}
\mathbb{E}[g^{(d)}_{\mathcal{V}}(\omega)^{-\frac{\gamma}{1-\gamma}}]<2^N,
\end{eqnarray*}
for $d=2,3$.
\begin{proof}
 It is easy to see, by the independence of the functions $G^{(d)}_{J_{v_i}}(\omega)$ for different $i$'s, that 
 \begin{eqnarray*}
\mathbb{E}[g^{(d)}_{\mathcal{V}}(\omega)^{-\frac{\gamma}{1-\gamma}}]=\left( \mathbb{E}\left[\exp\left(-\frac{\gamma}{1-\gamma}F( G^{(d)}_{J_0}(\omega))\right)\right]\right)^N.
\end{eqnarray*}
We proceed by estimating separately the expectation in the cases of $d=2,3$.
\vskip 2mm
CASE $d=2$:  We have the bound on the expectation
\begin{eqnarray}\label{2}
1+e^{\frac{\gamma}{1-\gamma}K_1}\mathbb{P}\left[ \delta_L\sum_{x\in J_0}\overline\omega_x>e^{K_2}\right]
&\leq& 1+e^{\frac{\gamma}{1-\gamma}K_1-2K_2}\,\mathbb{E}[\overline\omega^2] \, |J_0|\,\delta_L^2\\
&=&1+2C_3 e^{\frac{\gamma}{1-\gamma}K_1-2K_2}\,\mathbb{E}[\overline\omega^2] \nonumber\\
&<&2,\nonumber
\end{eqnarray}
for $K_2$ large enough.
\vskip 2mm
CASE $d=3$: We have the bound on the expectation
\begin{eqnarray*}
1+e^{\frac{\gamma}{1-\gamma}K_1}\mathbb{P}\left[ \sum_{y,z\in J_0}V_{y,z}\overline\omega_y\overline\omega_z>e^{K_2}\right]
&\leq&
1+e^{\frac{\gamma}{1-\gamma}K_1-2K_2} \mathbb{E}\left[ \left(\sum_{y,z\in J_0}V_{y,z}\overline\omega_y\overline\omega_z\right)^2\right]\\
&=&1+e^{\frac{\gamma}{1-\gamma}K_1-2K_2} \mathbb{E}[\overline\omega^2]^2 \sum_{y,z\in J_0}V^2_{y,z}\\
&\leq&1+CC_3^2C_4^2\,e^{\frac{\gamma}{1-\gamma}K_1-2K_2} \mathbb{E}[\overline\omega^2]^2\nonumber\\
&<&2,
\end{eqnarray*}
for $K_2$ large enough.
\end{proof}
\end{proposition}
We continue by estimating the second term of \eqref{Holder}. The first part of this estimate is identical for both $d=2,3$.

\begin{proposition}\label{2nd Holder}
For any $\epsilon >0$ we can choose $L$ large enough such that
\begin{eqnarray}\label{Holder 2 ineq}
&&\sum_{\mathcal{V}}\left(\mathbb{E}_{\mathcal{V}}\left[ \,\hat{B}_\omega(NL;\cap_{i=1}^N \hat{X}_{iL}\in I^{C_1}_{v_i} )\right]\right)^\gamma\nonumber\\
&&\qquad\qquad\qquad  \leq  \left(C \sum_{v}\left( \max_{x\in I^{C_2}_0} \, \mathbb{E}_{J_0} \hat{B}_{x,\omega}(L;\hat{X}_L\in I^{C_1}_v)\right)^\gamma +\epsilon\right)^N,
\end{eqnarray}
for a positive constant $C$.
\end{proposition}
\begin{proof}
We start by
\begin{eqnarray}\label{decomposition}
\mathbb{E}_{\mathcal{V}}\hat{B}_\omega(NL;\cap_{i=1}^N \hat{X}_{iL}\in I^{C_1}_{v_i})&=&\mathbb{E}_{\mathcal{V}} \sum_{x_i\in I^{C_1}_{v_i};i=1,2,\dots,N} e^{\overline{m}_B^aL} E^\kappa\left[ e^{-\beta\sum_{x}\overline\omega_x\local_{T_L}(x)};\hat{X}_{L}=x_{1},\,\text{Br}(L)\right]\nonumber\\
       &&\qquad\qquad\times\prod_{i=1}^{N-1} e^{\overline{m}_B^aL} E_{x_i}^\kappa\left[ e^{-\beta\sum_{x}\overline\omega_x\local_{T_L}(x)};\hat{X}_{L}=x_{i+1}\,\right].
\end{eqnarray}
The terms of the product are not independent and a priori we cannot interchange the product and the $\mathbb{E}_{\mathcal{V}}$ . We can recover the independence by looking at when is the last time the path starting from $x_i\in I^{C_1}_{v_i}$ lies on the hyperplane $\mathcal{H}_{iL}:=\{x\colon x\cdot \hat{e}_1=iL\}$ and then bound below the potential $\beta\overline{\omega}$ of the sites visited by this segment of the walk by $-\beta\mathbb{E}[\omega]$. We then have
\begin{eqnarray}\label{maximum}
&&e^{\overline{m}_B^aL} E^\kappa_{x_i}\left[ e^{-\beta\sum_{x}\overline\omega_x\local_{T_L}(x)};\hat{X}_{L}=x_{i+1}\,\right] \nonumber\\
&&=\sum_{\tilde{x}\in \mathcal{H}_{iL}}\sum_{M=1}^\infty
e^{\overline{m}_B^aL} E_{x_i}^\kappa\left[ e^{-\beta\sum_{x}\overline\omega_x\local_{T_L}(x)};X_M=\tilde{x}\,\right] \overline{B}_{\tilde{x},\omega}(L;\hat{X}_{L}=x_{i+1})\nonumber\\
&&\leq e^{\overline{m}_B^aL}\sum_{\tilde{x}\in \mathcal{H}_{iL}}\sum_{M=1}^\infty e^{\beta\mathbb{E}[\omega] M}P^\kappa_{x_i}(X_M=\tilde{x})\, \overline{B}_{\tilde{x},\omega}(L;\hat{X}_L=x_{i+1}), 
\end{eqnarray}
for $i=1,\dots,N-1$. Since $x_i,\tilde{x}\in \mathcal{H}_{iL}$, we can use the identity  $e^{\beta\mathbb{E}[\omega] M}P^\kappa_{x_i}(X_M=\tilde{x})=e^{-\lambda M}P_{x_i}(X_M=\tilde{x})=e^{-\lambda M}P(X_M^{(1)}=0)$ into \eqref{maximum}, to get that \eqref{decomposition} is bounded by
\begin{eqnarray*}
   \prod_{i=0}^{N-1} e^{\overline{m}_B^aL} \max_{x\in I^{C_1}_{v_i}} \sum_{\tilde{x}\in \mathcal{H}_{iL}}\sum_{M=1}^\infty e^{-\lambda M} P_x(X_M=\tilde{x})\,\mathbb{E}_{J_{v_i}} \overline{B}_{\tilde{x},\omega}(L;\hat{X}_L\in I^{C_1}_{v_{i+1}}),
\end{eqnarray*}
and therefore the left hand side of \eqref{Holder 2 ineq} can be bounded by
\begin{eqnarray}\label{fractional_prod}
&&\\
\,\,\left( \sum_{v}  \left( e^{\overline{m}_B^aL}\max_{x\in I^{C_1}_0} 
\sum_{\tilde{x}\in \mathcal{H}_{0}}\sum_{M=1}^\infty e^{-\lambda M} P_x(X_M=\tilde{x})\,\mathbb{E}_{J_0} \overline{B}_{\tilde{x},\omega}(L;\hat{X}_L\in I^{C_1}_v)\right)^\gamma \,\right)^N.\nonumber
\end{eqnarray}
 To further estimate this, we decompose the term inside the "$\max_{x\in I^{C_1}_0}$" as follows
\begin{eqnarray}\label{decomposition2}
&&\sum_{M=1}^\infty\left( \sum_{\tilde{x}\in\mathcal{H}_{0}\colon |\tilde{x}|\leq C_2\sqrt{L}} + \sum_{\tilde{x}\in\mathcal{H}_{0}\colon  |\tilde{x}|>C_2\sqrt{L}} \right)
        e^{-\lambda M} P_x(X_M=\tilde{x})\, \mathbb{E}_{J_0} \hat{B}_{\tilde{x},\omega}(L;\hat{X}_L\in I^{C_1}_v).
\end{eqnarray}
 Notice that since   $ g^{(d)}_{J_0}(\omega)\leq 1$ we have that
  \begin{eqnarray}\label{bound1}
   \mathbb{E}_{J_0} \hat{B}_{\tilde{x},\omega}(L;\hat{X}_L\in I^{C_1}_v )
    &\leq& 
    \mathbb{E} \hat{B}_{\tilde{x},\omega}(L;\hat{X}_L\in I^{C_1}_v)\nonumber\\
    &=&
    \mathbb{E} \hat{B}_\omega(L;\hat{X}_L\in I^{C_1}_v-\tilde{x}).
\end{eqnarray}
Using this, \eqref{decomposition2} and the fractional inequality $(a+b)^\gamma\leq a^\gamma+b^\gamma, \gamma\leq 1$, we get that the part of \eqref{fractional_prod} inside the $N$ power is bounded by the sum of the terms
\begin{eqnarray}\label{first_term}
\sum_{v}\left(  \max_{x\in I^{C_1}_0}\sum_{M=1}^\infty \,\sum_{\tilde{x}\in\mathcal{H}_{0}\colon |\tilde{x}|\leq C_2\sqrt{L}} e^{-\lambda M}P_x(X_M=\tilde{x})\, \mathbb{E}_{J_0} \hat{B}_{\tilde{x},\omega}(L;\hat{X}_L\in I^{C_1}_v)\right)^\gamma,
\end{eqnarray}
and
\begin{eqnarray}\label{second_term}
\sum_{v}\left(  \max_{x\in I^{C_1}_0} \sum_{M=1}^\infty \,\sum_{ \tilde{x}\in\mathcal{H}_{0}\colon |\tilde{x}|>C_2\sqrt{L}} e^{-\lambda M} P_x(X_M=\tilde{x})\, \mathbb{E} \hat{B}_\omega(L;\hat{X}_L\in I^{C_1}_v-\tilde{x})\right)^\gamma.
\end{eqnarray}
BOUND ON \eqref{first_term}. Clearly, \eqref{first_term} is bounded by
\begin{eqnarray*}
\sum_{v}\left(  \max_{\tilde{x}\in I^{C_2}_0}\mathbb{E}_{J_0} \hat{B}_{\tilde{x},\omega}(L;\hat{X}_L\in I^{C_1}_v) \max_{x\in I^{C_1}_0}\sum_{M=1}^\infty \,\sum_{\tilde{x}\in\mathcal{H}_{0}\colon |\tilde{x}|\leq C_2\sqrt{L}} e^{-\lambda M} P_x(X_M=\tilde{x})\, \right)^\gamma.
\end{eqnarray*}
 Since $x,\tilde{x}\in\mathcal{H}_0$, it holds that 
\begin{eqnarray*}
\sum_{M=1}^\infty \sum_{\tilde{x}\in\mathcal{H}_{0}\colon |\tilde{x}|\leq C_2\sqrt{L}} e^{-\lambda M}P_x(X_{M}=\tilde{x})&\leq&
\sum_{M=1}^\infty e^{-\lambda M}P(X_M^{(1)}=0) :=\mu_0^{-1}.
\end{eqnarray*}
 Therefore, it follows that \eqref{first_term} is bounded by
\begin{eqnarray}\label{first_term_bound1}
&&\sum_{v}\left( \mu_0^{-1}\max_{x\in I^{C_2}_0} \, \mathbb{E}_{J_0} \hat{B}_{x,\omega}(L;\hat{X}_L\in I^{C_1}_v)\right)^\gamma\nonumber\\
&&\qquad\qquad\qquad= C \sum_{v}\left( \max_{x\in I^{C_2}_0} \, \mathbb{E}_{J_0} \hat{B}_{x,\omega}(L;\hat{X}_L\in I^{C_1}_v)\right)^\gamma,
\end{eqnarray}
where $C:=\mu_0^{-\gamma}$.

 BOUND ON \eqref{second_term}. We first use the usual fractional inequality to pass the $\gamma$ power inside the summations and then notice that the summation over $v\in \mathbb{Z}^{d}$ effectively eliminates the dependence on $\tilde{x}$ in the expectation. \eqref{second_term} is then bounded by
\begin{eqnarray}\label{second_term_bound2}
 &&\sum_{v}\left(  \mathbb{E} \hat{B}_\omega(L;\hat{X}_L\in I^{C_1}_v) \right)^\gamma \sum_{M=1}^\infty \sum_{x^\perp\in \mathbb{Z}^{d-1}\colon |x^\perp|>C_2\sqrt{L}} e^{-\gamma\lambda M} P(X^{(1)}_{M}=0)^\gamma P(X^\perp_M=x^\perp)^\gamma.
\end{eqnarray}
The process $(X^\perp_M)_{M\geq 1}$ is a simple random walk on $\mathbb{Z}^{d-1}$ and therefore we can use standard estimates to get $P(X^\perp_M=x^\perp)\leq C/ M^{(d-1)/2}e^{-C|x^\perp|^2/M}$.
We therefore have the simple computation
\begin{eqnarray*}
&&\sum_{M=1}^\infty \sum_{x^\perp\in \mathbb{Z}^{d-1}\colon |x^\perp|>C_2\sqrt{L}} e^{-\gamma \lambda M}P(X^{(1)}_{M}=0)^\gamma P(X^\perp_M=x^\perp)^\gamma \\
&&\qquad\qquad\leq C \sum_{M=1}^\infty \sum_{|x^\perp|>C_2\sqrt{L}}e^{-\gamma\lambda M} \frac{1}{ M^{\gamma (d-1)/2}} e^{-\gamma C |x^\perp|^2/M}\\
&&\qquad\qquad\leq C\left(\sum_{M=1}^{\sqrt{L}}+\sum_{M=\sqrt{L}+1}^\infty\right)e^{-\gamma\lambda M/2} e^{-\gamma CC_2^2 L/M}\leq Ce^{-\gamma C \sqrt{L}},
\end{eqnarray*}
and therefore \eqref{second_term_bound2} can be bounded by
\begin{eqnarray}\label{second_term_bound3}
Ce^{-\gamma C\sqrt{L}} \sum_{v}\left(   \mathbb{E} \hat{B}_\omega(L;\hat{X}_L\in I^{C_1}_v)\right)^\gamma<\epsilon,
\end{eqnarray}
with the last inequality valid for $L$ large enough, depending on $\epsilon$. Notice that here we used the fact that Proposition \ref{local CLT} guarantees the
uniform boundedness of $ \sum_{v}\left(   \mathbb{E} \hat{B}_\omega(L;\hat{X}_L\in I^{C_1}_v)\right)^\gamma$ in $L$.

The combination of \eqref{second_term_bound3} and \eqref{first_term_bound1} completes the proof of the proposition.
\end{proof}
The next proposition is the last step towards the proof of Theorem \ref{Thm 1}, part A.
\begin{proposition}\label{prop 9}
Consider $L$ chosen as 
\begin{eqnarray}\label{L,d=2 state}
L:=\left(\frac{e^{2K_2}}{|\overline\phi'(\beta)|}\right)^4,&&\qquad d=2,
\end{eqnarray} 
and
\begin{eqnarray}\label{l,d=3 state}
L:=\exp\left[\left(\frac{e^{2K_2}}{|\overline\phi'(\beta)|^2}\right)^2\right],\qquad d=3.
\end{eqnarray}
For any $\epsilon>0$ we can choose  $C_3$ large,
 $K_1$ large enough, depending on $\epsilon$ and $K_2$ large enough depending on $C_3, C_4, \mathbb{E}[\omega^2]$, $\epsilon$, $\beta$ (the dependence on $\beta$ is such that the length scale $L$ is large enough), such that
\begin{eqnarray}\label{prop 9 sum}
\sum_{v}\left( \max_{x\in I^{C_2}_0} \, \mathbb{E}_{J_0} \hat{B}_{x,\omega}(L;\hat{X}_L\in I^{C_1}_v)\right)^\gamma<\epsilon.
\end{eqnarray} 
\end{proposition}
The proof of this proposition requires different choices of the parameters in dimensions $2,3$ and is presented in the following subsections respectively. Before embarking into the proof of Proposition \ref{prop 9} we will show how Propositions \ref{Radon Nikodym}, \ref{2nd Holder}, \ref{prop 9} can be used to conclude the proof of Theorem \ref{Thm 1}, part A.
\vskip 4mm
{\bf Proof of Theorem \ref{Thm 1}, part A.} Using \eqref{Holder} into \eqref{difference masses} and using Propositions \ref{Radon Nikodym},\ref{2nd Holder}, \ref{prop 9} we have that 
\begin{eqnarray*}
\overline{m}_B^a-\overline{m}_B^q\leq \frac{1}{L} \gamma^{-1}\log (2^{(1-\gamma)}(C+1)\epsilon)<0,
\end{eqnarray*}
by choosing $\epsilon$ small enough.
Notice that the choice of $L$ in \eqref{L,d=2 state} and \eqref{l,d=3 state} provides also a bound on the gap between the annealed and quenched norms. When $\beta\sim 0$ is small, then $\overline\phi'(\beta)\sim -\beta \mathbb{E}[\overline\omega^2]$ and therefore the gap is bounded below by $O(\beta^{4})$ in $d=2$ and $O(\exp(-\beta^{-4}))$ in $d=3$. When $\beta$ is large the choice of $L$ being large imposes that $K_2$ must be chosen so that $\exp(2K_2):=\exp(2K'_2)|\overline\phi'(\beta)|$, in $d=2$, and $\exp(2K_2):=\exp(2K'_2)|\overline\phi'(\beta)|^2$, in $d=3$, with $K'_2$ large and therefore the bound on the gap in this case is $O(1)$. \hfill $\Box$

 \vskip 4mm
For the proof of Proposition \ref{prop 9} we will need the following notation
\begin{eqnarray}\label{notation 1}
&&\mathcal{X}_L^{C_3}=\{(X_.)\colon (X_n)_{1\leq n\leq T_L}\subset J_0\},
\end{eqnarray}
\begin{eqnarray}\label{notation 2}
&&\mathcal{B}_{v,L}^{C_3}:=\mathcal{X}_L^{C_3}\cap \text{Br}(L)\cap \{\hat{X}_L\in I^{C_1}_v\},
\end{eqnarray}
\begin{eqnarray}\label{notation 3}
&&\mathcal{B}_L^{C_3}:=\mathcal{X}_L^{C_3}\cap \text{Br}(L)
\end{eqnarray}

\vskip 2mm
\subsection{ Proof of Proposition \ref{prop 9} in dimension $d=2$.} In this case the coarse graining scale is chosen as
\begin{eqnarray}\label{L,d=2}
L:=\left(\frac{e^{2K_2}}{|\overline\phi'(\beta)|}\right)^4.
\end{eqnarray} 
Consider a parameter $R$ to be chosen later on and split the sum in \eqref{prop 9 sum} as follows
\begin{eqnarray}\label{d=2,1}
\sum_{v}\left( \max_{x\in I^{C_2}_0} \, \mathbb{E}_{J_0} \hat{B}_{x,\omega}(L;\hat{X}_L\in I^{C_1}_v)\right)^\gamma
&=&\sum_{v\colon |v|<R}\left( \max_{x\in I^{C_2}_0} \, \mathbb{E}_{J_0} \hat{B}_{x,\omega}(L;\hat{X}_L\in I^{C_1}_v)\right)^\gamma\nonumber\\
&&\quad+
\sum_{v\colon |v|\geq R}\left( \max_{x\in I^{C_2}_0} \, \mathbb{E}_{J_0} \hat{B}_{x,\omega}(L;\hat{X}_L\in I^{C_1}_v)\right)^\gamma\nonumber\\
&\leq&  R \max_{|v|<R}\max_{x\in I^{C_2}_0} \, \left(\mathbb{E}_{J_0} \hat{B}_{x,\omega}(L;\hat{X}_L\in I^{C_1}_v)\right)^\gamma\nonumber\\
&&\quad +\sum_{v\colon |v|\geq R}\left( \max_{x\in I^{C_2}_0} \, \mathbb{E} \hat{B}_{x,\omega}(L;\hat{X}_L\in I^{C_1}_v)\right)^\gamma\nonumber\\
&<&R \max_{|v|<R}\max_{x\in I^{C_2}_0} \, \left(\mathbb{E}_{J_0} \hat{B}_{x,\omega}(L;\hat{X}_L\in I^{C_1}_v)\right)^\gamma +\frac{\epsilon}{2},
\end{eqnarray}
where the last inequality follows from the local limit estimate of Proposition \ref{local CLT}, by choosing $R$ large enough. To estimate the first term of \eqref{d=2,1} we recall the definitions \eqref{notation 1}, \eqref{notation 2}, \eqref{notation 3} and write
\begin{eqnarray}\label{split}
&&\mathbb{E}_{J_0} \hat{B}_{x,\omega}(L;\hat{X}_L\in I^{C_1}_v)=
   \mathbb{E}_{J_0} \hat{B}_{x,\omega}(L;\mathcal{X}_L^{C_3},\hat{X}_L\in I^{C_1}_v)+ 
   \mathbb{E}_{J_0} \hat{B}_{x,\omega}(L;\overline{\mathcal{X}_L^{C_3}},\hat{X}_L\in I^{C_1}_v),
\end{eqnarray}
where for a set $A$ recall that we denote by $\overline{A}$ its complement. The estimate of \eqref{split} is based on the two following Lemmas
\begin{lemma}\label{first split}
We have the estimate
\begin{eqnarray*}
\mathbb{E}_{J_0} \hat{B}_{x,\omega}(L;\mathcal{X}_L^{C_3},\hat{X}_L\in I^{C_1}_v) <
4C_3\mathbb{E}[\omega^2] e^{-2K_2}+e^{-K_1},
\end{eqnarray*}
which can be made smaller than $\epsilon/4R$, if $K_1,K_2$ are chosen large enough. 
\end{lemma}
\begin{lemma}\label{second split}
For $C_3$ chosen large enough and $L$ chosen large enough, i.e. $K_2$ is chosen large enough, we have the estimate
\begin{eqnarray*}
\mathbb{E}_{J_0} \hat{B}_{x,\omega}(L;\overline{\mathcal{X}_L^{C_3}},\hat{X}_L\in I^{C_1}_v)<\frac{\epsilon}{4R}.
\end{eqnarray*}
\end{lemma}

Having established these two lemmas and inserting the corresponding estimates into \eqref{d=2,1}, the proof of Proposition \ref{prop 9} is completed. We therefore proceed to provide the proof of the lemmas.
 \vskip 4mm
{\it Proof of Lemma \ref{first split}}. 
For $x\in I^{C_2}_0$ we have

\begin{eqnarray}\label{44}
&& \,\mathbb{E}_{J_0} \hat{B}_{x,\omega}(L;\mathcal{X}_L^{C_3},\hat{X}_L\in I^{C_1}_v)=
     e^{\overline{m}_B^aL}\mathbb{E}\left[  g^{(2)}_{J_0}(\omega)   E^\kappa_x\left[ e^{-\beta\sum_x \overline\omega_x \local_{T_L}(x)} ; \mathcal{B}_{v,L}^{C_3} \right] \right]\nonumber\\
 &&\qquad \leq \,e^{\overline{m}_B^aL}\mathbb{E}\left[ E^\kappa_x\left[ e^{-\beta\sum_x \overline\omega_x \local_{T_L}(x)} ;\mathcal{B}_{v,L}^{C_3} \right] ; \delta_L\sum_{y\in J_0}\overline\omega_y<e^{K_2}\right] +\,e^{-K_1}\nonumber\\
 &&\qquad= \,e^{\overline{m}_B^aL}   E^\kappa_x\left[ \mathbb{E}\left[ e^{-\beta\sum_x \overline\omega_x \local_{T_L}(x)}; \delta_L\sum_{y\in J_0} \overline\omega_y <e^{K_2}\right] ;\mathcal{B}_{v,L}^{C_3}\right] +\,e^{-K_1}.
 \end{eqnarray}
 For a fixed path we define the measure $\mathbb{P}_X$  by
\begin{eqnarray*}
\frac{d\mathbb{P}_X}{d\mathbb{P}}:= \exp\left(\sum_x -\beta\overline\omega_x \local_{T_L}(x)+\overline\phi(\beta\ell_{T_L}(x))\right),
\end{eqnarray*} 
and we write \eqref{44} as
 \begin{eqnarray}\label{split1}
  &&\,e^{\overline{m}_B^aL} E^\kappa_x\left[  e^{-\overline\Phi_\beta(T_L)}\mathbb{P}_X\left[ \delta_L\sum_{y\in J_0} \overline\omega_y <e^{K_2}\right]; \mathcal{B}_{v,L}^{C_3}\right] +\,e^{-K_1}\nonumber\\
  &\leq& \,e^{\overline{m}_B^aL} E^\kappa_x\left[  e^{-\overline\Phi_\beta(T_L)} \mathbb{P}_X\left[ \delta_L\sum_{y\in J_0} \overline\omega_y <e^{K_2}\right]; \mathcal{B}_{L}^{C_3}\right] +\,e^{-K_1},
\end{eqnarray}
where in the last inequality we used the fact that $\mathcal{B}_{v,L}^{C_3}\subset \mathcal{B}_{L}^{C_3}$. 
\vskip 2mm
 We denote by $\mathcal{A}^{(1)}_L$ the event that $\mathbb{E}_X[\delta_L\sum_{y\in J_0} \overline\omega_y]>2e^{K_2}$. The first term of \eqref{split1} is then bounded by
\begin{eqnarray*}
 &&\,e^{\overline{m}_B^aL} E^\kappa_x\left[  e^{-\overline\Phi_\beta(T_L)} \mathbb{P}_X\left[ \delta_L\sum_{y\in J_0} \overline\omega_y -\mathbb{E}_X[\delta_L\sum_{y\in J_0} \overline\omega_y]<-e^{K_2}\right]; \mathcal{B}_{L}^{C_3}\cap \mathcal{A}^{(1)}_L
      \right]\\
  &+&
  e^{\overline{m}_B^aL} E^\kappa_x\left[  e^{-\overline\Phi_\beta(T_L)} ; \mathcal{B}_{L}^{C_3} \cap \overline{\mathcal{A}}^{(1)}_L\right], 
\end{eqnarray*}
and we can use Chebyshev's inequality to bound this by
\begin{eqnarray}\label{split2}
&&e^{-2K_2} e^{\overline{m}_B^aL} E^\kappa_x\left[ e^{-\overline\Phi_\beta(T_L)} \mathbb{E}_X\left[  \left(\delta_L \sum_{y\in J_0}(\overline\omega_y -\mathbb{E}_X[\overline\omega_y])\right)^2\right]; \mathcal{B}_{L}^{C_3}\cap \mathcal{A}^{(1)}_L
      \right]\nonumber\\
  &+&
  e^{\overline{m}_B^aL} E^\kappa_x\left[  e^{-\overline\Phi_\beta(T_L)}; \mathcal{B}_{L}^{C_3} \cap \overline{\mathcal{A}}^{(1)}_L\right].
\end{eqnarray}
{\bf Estimate on the first term of \eqref{split2}}. We first notice that $(\overline\omega_y)_{y\in J_0}$ are independent under the measure $\mathbb{P}_X$. Then it is easy to conclude that the first term of \eqref{split2} equals 
\begin{eqnarray*}
&&e^{-2K_2}  \delta_L^2 \sum_{y\in J_0}  e^{\overline{m}_B^aL} E^\kappa_x\left[ e^{-\overline\Phi_\beta(T_L)} \mathbb{E}_X\left[ \left(\overline\omega_y -\mathbb{E}_X[\overline\omega_y]\right)^2\right]; \mathcal{B}_{L}^{C_3}\cap \mathcal{A}^{(1)}_L\right]\\
&\leq& e^{-2K_2}  \delta_L^2 \sum_{y\in J_0}  e^{\overline{m}_B^aL} E^\kappa_x\left[ e^{-\overline\Phi_\beta(T_L)} \mathbb{E}_X\left[ \overline\omega^2_y \right]; \mathcal{B}_{L}^{C_3}\cap \mathcal{A}^{(1)}_L
      \right]\\
&=& e^{-2K_2}  \delta_L^2 \sum_{y\in J_0}  e^{\overline{m}_B^aL} E^\kappa_x\left[ e^{-\overline\Phi_\beta(T_L)} \mathbb{E}\left[ \overline\omega^2_y e^{-\beta\overline\omega_y\ell_{T_L}(y)+\overline\phi(\beta\ell_{T_L}(y))}\right]; \mathcal{B}_{L}^{C_3}\cap \mathcal{A}^{(1)}_L \right]\\
&\leq& 2 e^{-2K_2} \mathbb{E}[\omega^2] \,\,\delta_L^2 |J_0| \,\,e^{\overline{m}_B^aL} E^\kappa_x\left[ e^{-\overline\Phi_\beta(T_L)};\mathcal{B}_{L}^{C_3}\cap \mathcal{A}^{(1)}_L \right],
\end{eqnarray*}
where in the last inequality we used the fact that $\overline\omega_x\geq-\mathbb{E}[\omega_x]$, or $x\in \mathbb{Z}^d$, together with Harris-FKG implies easily that 
$\mathbb{E}\left[ \overline\omega^2_y e^{-\beta\overline\omega_y\ell_{T_L}(y)+\overline\phi(\beta\ell_{T_L}(y))}\right]\leq 
2\mathbb{E}[\omega^2]  .$
Finally using the fact that $e^{\overline{m}_B^aL }E_x\left[ e^{-\overline\Phi_\beta(T_L)};\,\text{Br}(L)\right]\leq 1$, see Proposition \ref{subadditivity}, and the fact that $\delta_L^2\, |J_0|=2C_3$ we have the following bound on the first term of \eqref{split2} 
\begin{eqnarray}\label{estimate on 46.1}
&&\,\,e^{\overline{m}_B^aL} E^\kappa_x\left[ e^{-\overline\Phi_\beta(T_L)} \mathbb{E}_X\left[  \left(\delta_L \sum_{y\in J_0}(\overline\omega_y -\mathbb{E}_X[\overline\omega_y])\right)^2\right]; \mathcal{B}_{L}^{C_3}\cap \mathcal{A}^{(1)}_L
      \right] 
   <4C_3\mathbb{E}[\omega^2]\,e^{-2K_2},
\end{eqnarray}
which can be made small by choosing $K_2$ large enough depending on $C_3,\mathbb{E}[\omega^2]$.
\vskip 2mm
{\bf  Estimate on the second term of \eqref{split2}.}  We first compute
\begin{eqnarray*}
\mathbb{E}_X[\delta_L\sum_{y\in J_0} \overline\omega_y]&=&\delta_L\sum_{y\in J_0} \mathbb{E}_X[\overline\omega_y]\\
&=&\delta_L \sum_{y\in J_0} \mathbb{E}\left[\overline\omega_y e^{-\beta\overline\omega_y\ell_{T_L}(y)+\overline\phi(\beta\ell_{T_L}(y))}\right]\\
&=&\delta_L\sum_{y\in J_0} \overline\phi'(\beta\ell_{T_L}(y))\,1_{\ell_{T_L}(y)>0}\\
&\geq&\overline\phi'(\beta)\delta_L \sum_{y\in J_0}1_{\ell_{T_L}(y)>0}\\
&\geq&\overline\phi'(\beta)\delta_L L\,1_{\mathcal{X}_L^{C_3}}\\
&=&e^{2K_2}\,1_{\mathcal{X}_L^{C_3}},
\end{eqnarray*}
where in the first inequality we used the concavity of the log-moment generating function $\overline\phi(\cdot)$, as this is defined in \eqref{log-mom}, and the fact that $\delta_L:=-L^{-3/4}$ is negative. 
 Since on $\mathcal{B}_L^{C_3}$ it holds that $1_{\mathcal{X}_L^{C_3}}=1$, we have that $e^{2K_2}\,1_{\mathcal{X}_L^{C_3}}>2e^{K_2}$ and therefore the second term of \eqref{split2} vanishes.

This fact together with \eqref{estimate on 46.1}, \eqref{split2} and \eqref{44} imply that
 \begin{eqnarray}\label{d=2,2}
 \mathbb{E}_{J_0} \hat{B}_{x,\omega}(L;\mathcal{X}_L^{C_3},\hat{X}_L\in I^{C_1}_v)<
4C_3\mathbb{E}[\omega^2] e^{-2K_2}+e^{-K_1} ,
\end{eqnarray}
and this completes the estimate of the first term of Lemma \ref{first split}.\hfill $\Box$

\vskip 4mm
{\it Proof of Lemma \ref{second split}.} 
 The term on the left hand side of the inequality is bounded by $\mathbb{E}\hat{B}_\omega(L;\overline{\mathcal{X}_L^{C_3}} )$. We further have
\begin{eqnarray*}
\mathbb{E}\hat{B}_\omega(L;\overline{\mathcal{X}_L^{C_3}} )&=& e^{\overline{m}_B^aL} E^\kappa\left[e^{-\overline\Phi_\beta(T_L)};\text{Br}(L)\cap\overline{\mathcal{X}_L^{C_3}} \right]\\
&=&\sum_{n=1}^\infty\sum_{L_1+\cdots+L_n=L}e^{\overline{m}_B^aL} E^\kappa\left[ e^{-\overline\Phi_\beta(T_L)}; \overline{\mathcal{X}_L^{C_3}}\bigcap\cap_{i=1}^n\text{Ir}(L_i)\right]\\
&=&\left(\sum_{|n-\mu^{-1}L|<{\epsilon_0} L}+\sum_{|n-\mu^{-1}L|>{\epsilon_0} L}\right)\sum_{L_1+\cdots+L_n=L}e^{\overline{m}_B^aL} E^\kappa\left[ e^{-\overline\Phi_\beta(T_L)}; \overline{\mathcal{X}_L^{C_3}}\bigcap\cap_{i=1}^n\text{Ir}(L_i)\right]\\
&:=&I+II,
\end{eqnarray*}
where $\mu$ is defined in \eqref{mu} and ${\epsilon_0}$ is fixed satisfying $\epsilon_0\mu<\epsilon(\gamma_0,\gamma_1,\gamma_2)$, as in Proposition \ref{BS}.

{\bf Control on $II$.} We have that
\begin{eqnarray*}
II&\leq&\sum_{n\colon |n-\mu^{-1}L|>{\epsilon_0} L}\,\sum_{L_1+\cdots+L_n=L}e^{\overline{m}_B^aL} E^\kappa\left[ e^{-\overline\Phi_\beta(T_L)}; \,\cap_{i=1}^n\text{Ir}(L_i)\right]\\
&=&\sum_{n\colon |n-\mu^{-1}L|>{\epsilon_0} L}p_n^\beta(L),
\end{eqnarray*}
where $p_n^\beta(L):=\sum_x p_n^\beta(L,x)$, with the latter defined in Proposition \ref{local CLT}.
Using Proposition \ref{local CLT} we have that
\begin{eqnarray*}
II\leq \sum_{n\colon |L-n\mu|>{\epsilon_0} \mu L} C e^{-\tilde\delta_{\epsilon_0\mu} |L-n\mu|}<C_{\epsilon_0} e^{-\tilde\delta_{\epsilon_0\mu}\, {\epsilon_0} \mu\, L}.
\end{eqnarray*}

{\bf Control on I.} We have 
\begin{eqnarray*}
I&\leq&\sum_{n\colon |n-\mu^{-1}L|<{\epsilon_0} L}\,\sum_{L_1+\cdots+L_n=L}e^{\overline{m}_B^aL} E^\kappa\left[ e^{-\overline\Phi_\beta(T_L)}; \,\,\bigcap_{i=1}^n\text{Ir}(L_i)\,\bigcup_{i=1}^n\,\{|X^\perp(\tau_i)|>\frac{C_3}{2}\sqrt{L}\}\right]\\
&&+\sum_{n\colon |n-\mu^{-1}L|<{\epsilon_0} L}\,\sum_{L_1+\cdots+L_n=L}e^{\overline{m}_B^aL} E^\kappa\Big[ e^{-\overline\Phi_\beta(T_L)}; \,\,\bigcap_{i=1}^n\text{Ir}(L_i)\,\bigcap_{i=1}^n\{|X^\perp(\tau_i)|<\frac{C_3}{2}\sqrt{L}\}\\
&& \qquad\qquad\qquad\qquad\qquad\qquad\qquad\qquad\bigcup_{i=1}^n \{\sup_{\tau_i<m<\tau_{i+1}}|X_m^\perp|>C_3\sqrt{L}\}\Big]\\
&:=&I_i+I_{ii}.
\end{eqnarray*}
To bound the term $I_{ii}$ notice that the event $\{|X^\perp(\tau_i)|<\frac{C_3}{2}\sqrt{L}\} \,\cap\, \{|X^\perp(\tau_{i+1})|<\frac{C_3}{2}\sqrt{L}\}\bigcap  \{\sup_{\tau_i<m<\tau_{i+1}}|X_m^\perp|>C_3\sqrt{L}\}$ implies that $\tau_{i+1}-\tau_i>\frac{C_3}{2}\sqrt{L}$.
Therefore the term $I_{ii}$ is bounded above by
\begin{eqnarray*}
P^\beta\left(\bigcup_{i=1}^{(\mu^{-1}+{\epsilon_0})L} \{\tau_i-\tau_{i-1}>\frac{C_3}{2}\sqrt{L}\}\right)&\leq&
 (\mu^{-1}+{\epsilon_0})L \,P^\beta\left(\tau_1>\frac{C_3}{2}\sqrt{L}\right)\\
 &\leq&(\mu^{-1}+{\epsilon_0})L e^{-\frac{1}{4}\rho_1 C_3\sqrt{L}}.
\end{eqnarray*}
since $\tau_1$ has exponential moments under $P^\beta$, as this is implied by Proposition \ref{exponential_moments}.
Finally, we bound the term $I_i$. It is easy to see that 
\begin{eqnarray}\label{estimateI_i}
I_i&\leq&P^\beta\left( \bigcup_{i=1}^{(\mu^{-1}+{\epsilon_0})L}\{ |X^\perp(\tau_i)|>\frac{C_3}{2}\sqrt{L}\}\right).
\end{eqnarray}
Since the increments of $X^\perp(\tau_i)$ are independent under $P^\beta$ with exponential moments, see Proposition \ref{exponential_moments} , it follows from the standard theory of random walks that \eqref{estimateI_i} can be made arbitrarily small if $C_3$ is chosen large enough. 

Summing up, we can choose $C_3$ and $L$ large enough (e.g. choosing $K_2$ large), so that
\begin{eqnarray}\label{d=2,3}
\mathbb{E}_{J_0} \hat{B}_{x,\omega}(L;\overline{\mathcal{X}_L^{C_3}},\hat{X}_L\in I^{C_1}_v)
 \leq 
\mathbb{E}\hat{B}_\omega(L;\overline{\mathcal{X}_L^{C_3}} )
 <\frac{\epsilon}{4R},
\end{eqnarray}
for every $v$. This concludes the proof of the Lemma.\hfill$\Box$
\subsection{ Proof of Proposition \ref{prop 9} in dimension $d=3$.} 
In this case we recall the definitions \eqref{d=3},\eqref{V}. We will be choosing $L$ such that
\begin{eqnarray}\label{Ld3}
\log L:=\left( \frac{e^{2K_2}}{|\overline\phi'(\beta)|^2}\right)^2.
\end{eqnarray}
The first steps are the same as in the $d=2$ case. In particular, inequality \eqref{d=2,1} and decomposition \eqref{split} are still valid. Lemma \ref{second split} is still valid and used to estimate the second term in \eqref{split}. We therefore need to control the first term of \eqref{split}. This is done in the following lemma, which is the analogue of Lemma \ref{first split}.
\begin{lemma}\label{first split 2}
Given $\delta_0>0$ we can choose $C_4$ large enough, depending on $\delta_0$ and also $K_2$ large enough, such that
\begin{eqnarray*}
\mathbb{E}_{J_0} \hat{B}_{x,\omega}(L;\mathcal{X}_L^{C_3},\hat{X}_L\in I^{C_1}_v) <
CC_3^2C_4^2 \mathbb{E}[\omega^2]^2 e^{-2K_2}+2\delta_0+e^{-K_1}.
\end{eqnarray*}
\end{lemma}
Once this lemma is established Proposition \ref{prop 9} follows by choosing $K_1$ large, $\delta_0$ small, $C_4$ large, depending on $\delta_0$ and $K_2$ large depending on $C_3,C_4$ and $\mathbb{E}[\omega^2]$. We are now left with the proof of Lemma \ref{first split 2}.
\vskip 4mm
{\it Proof of Lemma \ref{first split 2}}. The beginning of the proof is identical to that of Lemma \ref{first split} up to inequality \eqref{44}, which now writes as
\begin{eqnarray*}
 &&\,\mathbb{E}_{J_0} \hat{B}_{x,\omega}(L;\mathcal{X}_L^{C_3},\hat{X}_L\in I^{C_1}_v) \leq
\,e^{\overline{m}_B^aL} E^\kappa_x\left[ e^{-\overline\Phi_\beta(T_L)}\, \mathbb{P}_X\left[  \sum_{y,z\in J_0} V_{y,z}\overline\omega_y\overline\omega_z < e^{K_2}\right]; \mathcal{B}_{L}^{C_3}\right] +\,e^{-K_1}.
\end{eqnarray*}
Denote by 
$\mathcal{A}^{(2)}_{L}$ the event that $\mathbb{E}_X[\sum_{y,z\in J_0} V_{y,z}\overline\omega_y\overline\omega_z]>2e^{K_2} $.
Then the first term in the right hand side of the above inequality is bounded by
\begin{eqnarray}\label{fractional,3}
&&e^{\overline{m}_B^aL} E^\kappa_x\left[   e^{-\overline\Phi_\beta(T_L)}\,\mathbb{P}_X\left[  \sum_{y,z\in J_0} V_{y,z}\overline\omega_y\overline\omega_z < e^{K_2}\right];\,\mathcal{B}_{L}^{C_3}\cap \mathcal{A}^{(2)}_{L} \right]\nonumber\\
&&    +e^{\overline{m}_B^aL} E^\kappa_x\left[  e^{-\overline\Phi_\beta(T_L)} ;\, \mathcal{B}_{L}^{C_3} \cap \overline{\mathcal{A}}^{(2)}_{L}\right].
\end{eqnarray}
{\bf Estimate on the first term of \eqref{fractional,3}.} We subtract the quantity $\mathbb{E}_X[\sum_{y,z\in J_0} V_{y,z}\overline\omega_y\overline\omega_z]$ from both sides in the event $\{ \sum_{y,z\in J_0} V_{y,z}\overline\omega_y\overline\omega_z<\exp(K_2)\}$ in the first term of \eqref{fractional,3} and we use Chebyshev's inequality to obtain the upper bound
\begin{eqnarray}\label{variance,3}
&&\\
&& e^{-2K_2}e^{\overline{m}_B^aL} E^\kappa_x\left[   e^{-\overline\Phi_\beta(T_L)}\,\,\mathbb{E}_X\left ( \sum_{y,z\in J_0} V_{y,z}\overline\omega_y\overline\omega_z- \mathbb{E}_X[\sum_{y,z\in J_0} V_{y,z}\overline\omega_y\overline\omega_z]\right)^2;\,\mathcal{B}_{L}^{C_3}\cap \mathcal{A}^{(2)}_{L} \right].\nonumber
\end{eqnarray} 
Notice that 
$$\mathbb{E}_X[\sum_{y,z\in J_0} V_{y,z}\overline\omega_y\overline\omega_z]=\sum_{y,z\in J_0} V_{y,z} \,\,
       \mathbb{E}[\overline\omega_ye^{-\beta\overline\omega_y\ell_{T_L}(y)+\overline\phi(\beta\ell_{T_L}(y))}] 
       \mathbb{E}[\overline\omega_ze^{-\beta\overline\omega_z\ell_{T_L}(z)+\overline\phi(\beta\ell_{T_L}(z))}].$$
Recall that $V_{y,y}=0$. Write the quantity inside the square of \eqref{variance,3} as
\begin{eqnarray*}
&&\sum_{y,z\in J_0} V_{y,z} 
       \left(\overline\omega_y- \mathbb{E}[\overline\omega_ye^{-\beta\overline\omega_y\ell_{T_L}(y)+\overline\phi(\beta\ell_{T_L}(y))}] 
       \right) 
       \left(\overline\omega_z-\mathbb{E}[\overline\omega_ze^{-\beta\overline\omega_z\ell_{T_L}(z)+\overline\phi(\beta\ell_{T_L}(z))}]\right)\\
&&+\quad 2\sum_{y,z\in J_0} V_{y,z} \,\,
       \mathbb{E}[\overline\omega_ye^{-\beta\overline\omega_y\ell_{T_L}(y)+\overline\phi(\beta\ell_{T_L}(y))}] 
       \left(\overline\omega_z-\mathbb{E}[\overline\omega_ze^{-\beta\overline\omega_z\ell_{T_L}(z)+\overline\phi(\beta\ell_{T_L}(z))}]\right),
\end{eqnarray*}
and proceed to the estimate
\begin{align}\label{variance 3.1}
 &\mathbb{E}_X\left ( \sum_{y,z\in J_0} V_{y,z}\overline\omega_y\overline\omega_z-\mathbb{E}_X[\sum_{y,z\in J_0} V_{y,z}\overline\omega_y\overline\omega_z]\right)^2\nonumber\\
 &\leq 2\mathbb{E}_X \left[ \left(\sum_{y,z\in J_0} V_{y,z} 
       \left(\overline\omega_y- \mathbb{E}[\overline\omega_ye^{-\beta\overline\omega_y\ell_{T_L}(y)+\overline\phi(\beta\ell_{T_L}(y))}] 
       \right) 
       \left(\overline\omega_z-\mathbb{E}[\overline\omega_ze^{-\beta\overline\omega_z\ell_{T_L}(z)+\overline\phi(\beta\ell_{T_L}(z))}]\right) 
       \right)^2\right] \nonumber\\
 &\,+\, 8\,       \mathbb{E}_X \left[ \left(\sum_{z\in J_0} \sum_{y\in J_0}V_{y,z} 
        \mathbb{E}[\overline\omega_ye^{-\beta\overline\omega_y\ell_{T_L}(y)+\overline\phi(\beta\ell_{T_L}(y))}] 
       \left(\overline\omega_z-\mathbb{E}[\overline\omega_ze^{-\beta\overline\omega_z\ell_{T_L}(z)+\overline\phi(\beta\ell_{T_L}(z))}]\right) 
       \right)^2\right].
\end{align}
We can now use the fact that $\mathbb{P}_X$ is a product measure and thus the $\omega'$s at different sites are independent. Recalling that $V_{y,y}=0$, we write \eqref{variance 3.1}  as
\begin{eqnarray}\label{variance 3.2}
&&2\,\sum_{y,z\in J_0} V^2_{y,z} 
       \prod_{x=y,z} \mathbb{E}\left[ e^{-\beta\overline\omega_x\ell_{T_L}(x) +\overline\phi(\beta\ell_{T_L}(x))} 
       \left(\overline\omega_x- \mathbb{E}[\overline\omega_xe^{-\beta\overline\omega_x\ell_{T_L}(x)+\overline\phi(\beta\ell_{T_L}(x))}] 
       \right)^2\right]\nonumber\\
&&+ 8\sum_{z\in J_0} \left(\sum_{y\in J_0}V_{y,z} \mathbb{E}\left[ \overline\omega_ye^{-\beta\overline\omega_y\ell_{T_L}(y) +\overline\phi(\beta\ell_{T_L}(y))}\right] \right)^2 \nonumber\\
&&\qquad\qquad\qquad\times\mathbb{E}\left[ e^{-\beta\overline\omega_z\ell_{T_L}(z) +\overline\phi(\beta\ell_{T_L}(z))} 
       \left(\overline\omega_z- \mathbb{E}[\overline\omega_z e^{-\beta\overline\omega_z\ell_{T_L}(z)+\overline\phi(\beta\ell_{T_L}(z))}] 
       \right)^2\right]\nonumber\\
&\leq&        2\,\sum_{y,z\in J_0} V^2_{y,z} 
       \prod_{x=y,z} \mathbb{E}\left[ \overline\omega_x^2 e^{-\beta\overline\omega_x\ell_{T_L}(x) +\overline\phi(\beta\ell_{T_L}(x))} \right]\nonumber\\
&&\qquad 8 \sum_{z\in J_0} \left(\sum_{y\in J_0}V_{y,z} \mathbb{E}\left[ \overline\omega_ye^{-\beta\overline\omega_y\ell_{T_L}(y) +\overline\phi(\beta\ell_{T_L}(y))}\right] \right)^2 
\mathbb{E}\left[ \overline\omega_z^2e^{-\beta\overline\omega_x\ell_{T_L}(x) +\overline\phi(\beta\ell_{T_L}(x))} \right].
\end{eqnarray}
By Harris-FKG and the fact that $\overline\omega_x\geq -\mathbb{E}[\omega_x]$, for $x\in \mathbb{Z}^{d}$, it is easy to conclude that $\mathbb{E}\left[ \overline\omega_x^2 e^{-\beta\overline\omega_x\ell_{T_L}(x) +\overline\phi(\beta\ell_{T_L}(x))} \right]\leq 2\mathbb{E}[\omega^2]$. Using once again the fact that $\overline\omega_x\geq -\mathbb{E}[\omega_x]$, we also obtain that
\begin{eqnarray} \label{variance 3.3}
\left(\sum_{y\in J_0}V_{y,z} \mathbb{E}\left[ \overline\omega_ye^{-\beta\overline\omega_y\ell_{T_L}(y) +\overline\phi(\beta\ell_{T_L}(y))}\right] \right)^2\leq 
\left(\sum_{y\in J_0}V_{y,z} 1_{\ell_{T_L}(y)>0}\right)^2 \mathbb{E}[ \omega]^2.
\end{eqnarray}
Using these two facts we bound \eqref{variance 3.2} by
\begin{eqnarray}\label{variance 3.4}
 8\,\mathbb{E}[\omega^2]^2 \,\sum_{y,z\in J_0} V^2_{y,z} + 
  16\,\mathbb{E}[\omega^2]^2 \,\sum_{z\in J_0} \left(\sum_{y\in J_0}V_{y,z} 1_{\ell_{T_L}(y)>0}\right)^2.
\end{eqnarray}
Clearly, from the choice of $V_{y,z}$ we have $\sum_{y,z\in J_0}V_{y,z}^2<CC_3^2C_4^2$ and therefore \eqref{variance 3.4}
is bounded by
\begin{eqnarray*}
C\, \mathbb{E}[\omega^2]^2 \left( C_3^2C_4^2+\sum_{z\in J_0}\left(\sum_{y\in J_0}V_{y,z} 1_{\ell_{T_L}(y)>0}\right)^2 \right).
\end{eqnarray*}
We can then bound \eqref{variance,3} by
\begin{eqnarray*}
C\, \mathbb{E}[\omega^2]^2 e^{-2K_2}\left(C_3^2C_4^2+ e^{\overline{m}_B^aL} E^\kappa_x\big[ e^{-\overline\Phi_\beta(T_L)} \sum_{z\in J_0} \big(\sum_{y\in J_0}V_{y,z} 1_{\ell_{T_L}(y)>0}\big)^2 ;\,\text{Br}(L)\big]\right).
\end{eqnarray*}
It remains to bound uniformly in $L$ the above expectation, which is done as follows. First, we expand the  square by summing up over $y,\tilde{y}\in J_0$ and then interchange the summations and use Cauchy-Schwartz:
\begin{eqnarray*}
&&\sum_{y,\tilde{y}\in J_0} \frac{e^{\overline{m}_B^aL}}{L^2\log L} E^\kappa_x\left[ e^{-\overline\Phi_\beta(T_L)} 1_{\ell_{T_L}(y),\ell_{T_L}(\tilde{y})>0};\,\text{Br}(L)\right]\\
 &&\qquad\qquad   \times\left( \sum_{z\in J_0} \frac{1_{|y^\perp-z^\perp|<C_4\sqrt{|y^{(1)}-z^{(1)}|}}}{|y^{(1)}-z^{(1)}|^2+1}\right)^{1/2}
             \left( \sum_{z\in J_0} \frac{1_{|\tilde{y}^\perp-z^\perp|<C_4\sqrt{|\tilde{y}^{(1)}-z^{(1)}|}}}{|\tilde{y}^{(1)}-z^{(1)}|^2+1}\right)^{1/2}\\
&\leq&C C_4^2  \sum_{y,\tilde{y}\in J_0} L^{-2}e^{\overline{m}_B^aL} E^\kappa_x\left[ e^{-\overline\Phi_\beta(T_L)} 1_{\ell_{T_L}(y),\ell_{T_L}(\tilde{y})>0};\,\text{Br}(L)\right]   \\
&=& C C_4^2    L^{-2}e^{\overline{m}_B^aL} E^\kappa_x\left[ e^{-\overline\Phi_\beta(T_L)} (\sum_{y\in J_0} 1_{\ell_{T_L}(y)>0})^2;\,\text{Br}(L)\right]  \\
&\leq&CC_4^2.
\end{eqnarray*}
To justify the last inequality we use Theorem \ref{BS} which implies that  with $P^\beta$ probability which is exponentially in $L$ close to one, the number of break points until $T_L$ will be close to $\mu^{-1}L$. Moreover by Proposition \ref{exponential_moments} the range of the path within an irreducible bridge has exponential moments and these two facts yield that $e^{\overline{m}_B^aL} E_x\left[ e^{-\overline\Phi_\beta(T_L)} (\sum_{y\in J_0} 1_{\ell_{T_L}(y)>0})^2\right]\leq C L^2$.

We finally get that \eqref{variance,3} is bounded by $CC_3^2C_4^2\mathbb{E}[\omega^2]^2\,e^{-2K_2}$. This gives the first term of the right hand side inequality of the Lemma.

It remains to show that the second term in \eqref{fractional,3} can be made smaller than $2\delta_0$. This is done in the following paragraph.

{\bf Estimate on the second term of \eqref{fractional,3}.} We start by using once again the observation that $\overline\phi'(\beta\ell_{T_L}(y))\leq \overline\phi'(\beta)\leq0$. This will lead to
\begin{eqnarray*}
\mathbb{E}_X[\sum_{y,z\in J_0} V_{y,z}\overline\omega_y\overline\omega_z]&=& \sum_{y,z\in J_0} V_{y,z}\overline\phi'(\beta\ell_{T_L}(y))\overline\phi'(\beta\ell_{T_L}(z))\\
&\geq&\overline\phi'(\beta)^2\sum_{y,z\in J_0} V_{y,z}1_{\{\ell_{T_L}(y),\ell_{T_L}(z)>0\}}\\
&=&\frac{\overline\phi'(\beta)^2}{L(\log L)^{1/2}}\sum_{y,z\in J_0}\frac{1_{|y^\perp-z^\perp|<C_4\sqrt{|y^{(1)}-z^{(1)}|}}}{|y^{(1)}-z^{(1)}|+1}\,1_{y\neq z}1_{\{\ell_{T_L}(y),\ell_{T_L}(z)>0\}}\\
&\geq&\frac{2\overline\phi'(\beta)^2}{L(\log L)^{1/2}}\sum_{L_1=0}^L\sum_{y\in J_0\cap\mathcal{H}_{L_1}}1_{\{\ell_{T_L}(y)>0\}} \sum_{L_2=L_1+1}^{L} \frac{1}{|L_2-L_1|+1}\times \\
&&\qquad\qquad\qquad\qquad\qquad
      \times\sum_{z\in J_0\cap\mathcal{H}_{L_2}} 1_{\{|y^\perp-z^\perp|<C_4\sqrt{|L_2-L_1|}\}}\,  1_{\{\ell_{T_L}(z)>0\}}.
\end{eqnarray*}
Since we are restricted on the set $\mathcal{B}_L^{C_3}$, the path stays within the box $J_0$, and so we can drop the restriction that $y,z\in J_0$. Recall that $X(S_{L_1})$ is the last hitting point of the hyperplane $\mathcal{H}_{L_1}$  and that $X(T_{L_2})$ is the first hitting point of the hyperplane $\mathcal{H}_{L_2}$. We then have
\begin{eqnarray*}
&&\frac{2\overline\phi'(\beta)^2}{L(\log L)^{1/2}}\sum_{L_1=0}^L\sum_{y\in \mathcal{H}_{L_1}}1_{\{\ell_{T_L}(y)>0\}} \sum_{L_2=L_1+1}^{L} \frac{1}{|L_2-L_1|+1}
      \sum_{z\in \mathcal{H}_{L_2}} 1_{\{|y^\perp-z^\perp|<C_4\sqrt{|L_2-L_1|}\}}\,  1_{\{\ell_{T_L}(z)>0\}}\\
     &&\qquad  \geq
           \frac{2\overline\phi'(\beta)^2}{L(\log L)^{1/2}}\sum_{L_1=0}^L \sum_{L_2=L_1+1}^{L} 
           \frac{1_{\{|X^\perp(T_{L_2})-X^\perp(S_{L_1})|<C_4\sqrt{|L_2-L_1|}\}} }{|L_2-L_1|+1} .
            \end{eqnarray*}

Let us denote the quantity on the right hand side of the above inequality by $\mathbb{Y}_L$. Inserting the above estimate into the second term of \eqref{fractional,3} we have that
\begin{eqnarray}\label{second part}
e^{\overline{m}_B^aL} E^\kappa_x\left[  e^{-\overline\Phi_\beta(T_L)} ;\, \mathcal{B}_{L}^{C_3} \cap \overline{\mathcal{A}}^{(2)}_{L}\right]
&\leq& 
 e^{\overline{m}_B^aL} E^\kappa_x\left[  e^{-\overline\Phi_\beta(T_L)}\, 1_{\{\mathbb{Y}_L<2e^{K_2}\}} ;\, \mathcal{B}_{L}^{C_3} \right] \nonumber\\
&<&
e^{\overline{m}_B^aL} E^\kappa_x\left[  e^{-\overline\Phi_\beta(T_L)}\, 1_{\{\mathbb{Y}_L<2e^{K_2}\}};\,\text{Br}(L) \right].
\end{eqnarray}
To proceed, notice that we have the bound $\mathbb{Y}_L\leq D(L):=\frac{2\overline\phi'(\beta)^2}{L(\log L)^{1/2}}\sum_{L_1=0}^L\sum_{L_2=L_1+1}^{L} (|L_2-L_1|+1)^{-1}$ and $D(L)\geq C \overline\phi'(\beta)^2 (\log L)^{1/2}$.  Moreover,
\begin{eqnarray}\label{Y upper}
e^{\overline{m}_B^aL} E^\kappa_x\left[  e^{-\overline\Phi_\beta(T_L)}\,\mathbb{Y}_L;\,\text{Br}(L) \right] &=&
   e^{\overline{m}_B^aL} E^\kappa_x\left[  e^{-\overline\Phi_\beta(T_L)}\,\mathbb{Y}_L;\, \{\mathbb{Y}_L<2^{-1} D(L)\},\, \text{Br}(L) \right] \nonumber\\
   &&\quad+ e^{\overline{m}_B^aL} E^\kappa_x\left[  e^{-\overline\Phi_\beta(T_L)}\,\mathbb{Y}_L;\, \{\mathbb{Y}_L\geq 2^{-1} D(L)\},\, \text{Br}(L) \right]\nonumber\\
   &\leq& 2^{-1} D(L) \,e^{\overline{m}_B^aL} E^\kappa_x\left[  e^{-\overline\Phi_\beta(T_L)}\,;\, \{\mathbb{Y}_L<2^{-1} D(L)\},\, \text{Br}(L) \right]\nonumber\\
   &&\quad+  D(L) \,e^{\overline{m}_B^aL} E^\kappa_x\left[  e^{-\overline\Phi_\beta(T_L)}\,;\, \{\mathbb{Y}_L\geq 2^{-1} D(L)\},\, \text{Br}(L) \right].
\end{eqnarray}
On the other hand we have
\begin{eqnarray}\label{Y lower}
&&  \qquad e^{\overline{m}_B^aL} E^\kappa_x\left[  e^{-\overline\Phi_\beta(T_L)}\,\mathbb{Y}_L;\,\text{Br}(L) \right] =
 D(L) \,e^{\overline{m}_B^aL} E^\kappa_x\left[  e^{-\overline\Phi_\beta(T_L)}\,;\, \text{Br}(L) \right]\\
 &&- \frac{2\overline\phi'(\beta)^2}{L(\log L)^{1/2}}\sum_{L_1=0}^L \sum_{L_2=L_1+1}^{L} 
           \frac{1 }{|L_2-L_1|+1}  \times\nonumber\\
           &&\qquad\times e^{\overline{m}_B^aL} E^\kappa_x\left[  e^{-\overline\Phi_\beta(T_L)}\,;{\{|X^\perp(T_{L_2})-X^\perp(S_{L_1})|\geq C_4\sqrt{|L_2-L_1|}\}}\, ,\text{Br}(L) \right]\nonumber\\
           &&\geq D(L) \,e^{\overline{m}_B^aL} E^\kappa_x\left[  e^{-\overline\Phi_\beta(T_L)}\,;\, \text{Br}(L) \right]-\delta_0  D(L),\nonumber
\end{eqnarray}
where the last inequality holds for $\delta_0$ small by Lemma \ref{delta}, below. The combination of \eqref{Y upper} and \eqref{Y lower}
gives that
\begin{eqnarray}\label{Y}
e^{\overline{m}_B^aL} E^\kappa_x\left[  e^{-\overline\Phi_\beta(T_L)}\,;\, \{\mathbb{Y}_L<2^{-1} D(L)\},\, \text{Br}(L) \right]< 2\delta_0.
\end{eqnarray}
 By the choice of $L$ in \eqref{Ld3} we have that $D(L)/2\geq C\overline\phi'(\beta)^2(\log L)^{1/2}=Ce^{2K_2}$. The latter is larger than $2e^{K_2}$, when $K_2$ is chosen large. Therefore \eqref{second part} implies via  \eqref{Y} that 
 \begin{eqnarray*}
 e^{\overline{m}_B^aL} E^\kappa_x\left[  e^{-\overline\Phi_\beta(T_L)} ;\, \mathcal{B}_{L}^{C_3} \cap \overline{\mathcal{A}}^{(2)}_{L}\right]
&\leq& 2\delta_0.
 \end{eqnarray*}
  This implies that the second term in  \eqref{fractional,3} can be made arbitrarily small. 
  
  To complete  we need to establish the following lemma
\begin{lemma}\label{delta}
Given $\delta_0>0$ we can choose $C_4$ large enough such that for any $0\leq L_1\leq L_2\leq L$ we have
\begin{eqnarray*}
e^{\overline{m}_B^aL} E^\kappa\left[  e^{-\overline\Phi_\beta(T_L)}\,;{\{|X^\perp(T_{L_2})-X^\perp(S_{L_1})|\geq C_4\sqrt{|L_2-L_1|}\}}\, ,\text{Br}(L) \right]<\delta_0.
\end{eqnarray*}  
\begin{proof}
By Proposition \ref{properties}, part $(ii)$, we have that
\begin{eqnarray*}
\overline\Phi_\beta(T_L)&\geq& \overline\Phi_\beta(T_{L_1})+\overline\Phi_\beta(S_{L_1},T_{L_2})+\overline\Phi_\beta(S_{L_2},T_L)\\
&&\qquad\qquad\quad-\beta\mathbb{E}[\omega](S_{L_1}-T_{L_1})- 
\beta\mathbb{E}[\omega](S_{L_2}-T_{L_2}),
\end{eqnarray*}
and therefore we have that
\begin{eqnarray}\label{delta 1}
&&e^{\overline{m}_B^aL} E^\kappa\left[  e^{-\overline\Phi_\beta(T_L)}\,;{\{|X^\perp(T_{L_2})-X^\perp(S_{L_1})|\geq C_4\sqrt{|L_2-L_1|}\}}\, ,\text{Br}(L) \right]\nonumber\\
&\leq& \quad\sum_{x_1,x_2\mathcal{H}_{L_1}}\sum_{x_3,x_4\mathcal{H}_{L_2}} 
     e^{\overline{m}_B^aL_1} E^\kappa\left[  e^{-\overline\Phi_\beta(T_{L_1})};\text{Br}(L_1),\,X(T_{L_1})=x_1\right] 
     \nonumber\\
&& \times\,          \sum_{n\geq 1} e^{-\lambda n}P_{x_1}\left( X(n)=x_2\right) 
\,\,e^{\overline{m}_B^a(L_2-L_1)} E^\kappa_{x_2}\Big[  e^{-\overline\Phi_\beta(T_{L_2-L_1})}\,; \nonumber\\
&&\qquad\qquad{\{|X^\perp(T_{L_2-L_1})|\geq C_4\sqrt{|L_2-L_1|}\}}\, 
                                                                           ,X(T_{L_2-L_1})=x_3,\,\text{Br}(L_2-L_1) \Big]\nonumber\\
 &&   \qquad \times\sum_{n\geq 1}e^{-\lambda n}P_{x_3}\left( X(n)=x_4\right)  \,\,
           e^{\overline{m}_B^a(L-L_2)} E^\kappa_{x_4}\left[  e^{-\overline\Phi_\beta(T_{L-L_2})};\text{Br}(L-L_2)\right] .                                                         
\end{eqnarray}
To further bound this we use Proposition \ref{subadditivity}, which assures that 
$$e^{\overline{m}_B^a(L-L_2)} E^\kappa_{x_4}\left[ e^{-\overline\Phi_\beta(T_{L-L_2})};\text{Br}(L-L_2)\right] \leq 1.$$
 We also have that $ \sum_{x_{j+1}\mathcal{H}_{L_j}}\sum_{n\geq 1}e^{-\lambda n}P_{x_j}\left( X(n)=x_{j+1}\right):=\mu_0^{-1}$, for $j=1,3$ and $x_j\in \mathcal{H}_{L_j}$. Setting for shorthand $l:=L_2-L_1$ we are lead to the following bound 
 \\ for \eqref{delta 1}
\begin{eqnarray}\label{n conv}
&&\mu_0^{-2}
     \, e^{\overline{m}_B^al} E^\kappa\left[  e^{-\overline\Phi_\beta(T_{l})}\,;{\{|X^\perp(T_{l})|\geq C_4\sqrt{l}\}},\,\text{Br}(l) \right]\nonumber\\
     &=&\mu_0^{-2} \hat{B}(l;|\hat{X}_l^\perp|>C_4 \sqrt{l}).
\end{eqnarray}
It is now immediate to conclude, using Proposition \ref{local CLT}, that \eqref{n conv} can be made small when $C_4$ is chosen large.

\end{proof}
\end{lemma}

\section{Path Localization}\label{proof of Thm 2}
In this section we will prove Theorem \ref{Thm 2}, that is that the measure
$\mu_{L,\omega}^{\beta,\lambda}(\cdot)$ defined in  \eqref{p-t-h-measure} develops atoms, whenever the annealed and quenched Lyapounov norms are different. It will be more convenient and equivalent to prove the analogous statement for the measure 
\begin{eqnarray*}
\hat\mu_{L,\omega}^{\beta,\lambda}(x):=\frac{\overline{B}_\omega(L;x)}{\overline{B}_\omega(L)},
\end{eqnarray*} 
where $\overline{B}_\omega(L;x)$ is a shorthand notation for $\overline{B}_\omega(L;X(T_L)=x)$.
 In other words we will prove that 
\begin{eqnarray}\label{localization}
\limsup_{L\to\infty}\sup_{x\in \mathcal{H}_L}\hat\mu_{L,\omega}^{\beta,\lambda}(x)>0,\qquad \mathbb{P}-a.s.,
\end{eqnarray}
whenever the annealed and quenched Lyapounov masses $\overline{m}_B^a$ and $\overline{m}_B^q$ are different.
 
Before proceeding with the proof let us give the heuristic argument. The symbols $\simeq$ and $\lesssim$ in this heuristic argument are meant to be interpreted as {\it almost equal to} and {\it asymptotically less than}, in the limit when the length scales $N,L$, are large.

 Suppose that the annealed and quenched Lyapounov masses are different, or equivalently that there is an $\epsilon_1>0$ such that 
\begin{eqnarray}\label{inequality}
\overline{m}_B^a+\epsilon_1<\overline{m}_B^q.
\end{eqnarray}
We then have that $\mathbb{P}-a.s.$, for $N$ large enough,
\begin{eqnarray}\label{heuristic}
\overline{m}_B^q &\simeq &-\frac{1}{NL}\log \overline{B}_\omega(NL)=-\frac{1}{N}\sum_{n=1}^{N}\frac{1}{L}\log \frac{\overline{B}_\omega(nL)}{\overline{B}_\omega((n-1)L)}\nonumber\\
&\leq& -\frac{1}{N}\sum_{n=1}^{N}\frac{1}{L}\log \sum_{x\in\mathcal{H}_{(n-1)L}} \frac{\overline{B}_\omega((n-1)L;x) }{\overline{B}_\omega((n-1)L)}  \overline{B}_{x,\omega}(L)\nonumber\\
&=& -\frac{1}{N}\sum_{n=1}^{N} \frac{1}{L}\log \sum_{x\in\mathcal{H}_{(n-1)L}}\hat\mu_{(n-1)L,\omega}^{\beta,\lambda}(x) \overline{B}_{x,\omega}(L).
\end{eqnarray}
Moreover, the inequality above is obtained by restricting the path not to backtrack once it reached level $\mathcal{H}_{(n-1)L}$. 
Notice that $\hat\mu_{(n-1)L,\omega}^{\beta,\lambda}(x) $ and $\overline{B}_{x,\omega}(L)$ are independent if $x\in \mathcal{H}_{(n-1)L}$. If \eqref{localization} is not valid, that is the measure  $\hat\mu_{(n-1)L,\omega}^{\beta,\lambda}$ does not develop atoms but it rather spreads out, then an ergodicity argument should imply that, for $n$ large
\begin{eqnarray*}
 \sum_{x\in\mathcal{H}_{(n-1)L}}\hat\mu_{(n-1)L,\omega}^{\beta,\lambda}(x) \overline{B}_{x,\omega}(L) \simeq \mathbb{E} \overline{B}_{\omega}(L)=\overline{B}(L).
\end{eqnarray*}
Then, by a standard Cesaro argument, we will have that
\begin{eqnarray*}
-\frac{1}{N}\sum_{n=1}^{N} \frac{1}{L}\log \sum_{x\in\mathcal{H}_{(n-1)L}}\hat\mu_{(n-1)L,\omega}^{\beta,\lambda}(x) \overline{B}_{x,\omega}(L)\simeq -\frac{1}{L}\log \overline{B}(L),
\end{eqnarray*}
for $N$ large.
Then \eqref{heuristic} would lead to
\begin{eqnarray*}
\overline{m}_B^q \lesssim -\frac{1}{L}\log \overline{B}(L)\simeq \overline{m}_B^a,
\end{eqnarray*}
when $L$ is large enough, which contradicts \eqref{inequality}. To make this argument rigorous we need essentially to make the {\it ergodicity argument} precise. 

To this end, we start by using the fact that $\overline{B}(L)\geq \mu_0 e^{-\overline{m}_B^aL}$ (from Proposition
 \ref{subadditivity}) in the first inequality below and Chebyshev's inequality in the second one to obtain the following estimate, for $L$ large depending on $\epsilon_1$,
\begin{eqnarray}\label{Cesaro}
&&\quad\mathbb{P}\left(-\frac{1}{NL}\sum_{n=0}^{N-1}\log \sum_{x\in \mathcal{H}_{nL}} \hat\mu_{nL,\omega}^{\beta,\lambda}(x) \overline{B}_{x,\omega}(L)>\overline{m}_B^a +\frac{\epsilon_1}{2}\right)\nonumber\\
&\leq& 
\mathbb{P}\left(-\frac{1}{NL}\sum_{n=0}^{N-1}\log \sum_{x\in \mathcal{H}_{nL}} \hat\mu_{nL,\omega}^{\beta,\lambda}(x) \overline{B}_{x,\omega}(L)>-\frac{1}{L}\log \overline{B}(L) +\frac{\epsilon_1}{4}\right)\nonumber\\
&=&\mathbb{P}\left(-\frac{1}{N}\sum_{n=0}^{N-1}\frac{1}{L}\log \sum_{x\in \mathcal{H}_{nL}} \hat\mu_{nL,\omega}^{\beta,\lambda}(x)\frac{ \overline{B}_{x,\omega}(L)}{\overline{B}(L)}>\frac{\epsilon_1}{4}\right)\nonumber\\
&\leq&\frac{4}{\epsilon_1} \,\,\mathbb{E}\left[ \left|\frac{1}{N}\sum_{n=0}^{N-1}\frac{1}{L}\log \sum_{x\in \mathcal{H}_{nL}} \hat\mu_{nL,\omega}^{\beta,\lambda}(x)\frac{ \overline{B}_{x,\omega}(L)}{\overline{B}(L)} \right| \right]\nonumber\\
&\leq&\frac{4}{\epsilon_1} \,\, \frac{1}{N} \left( \sum_{n=0}^{\delta_1 N}+\sum_{n=\delta_1 N+1}^{N-1}\right)
 \mathbb{E}\left[ \left|\frac{1}{L}\log \sum_{x\in \mathcal{H}_{nL}} \hat\mu_{nL,\omega}^{\beta,\lambda}(x)\frac{ \overline{B}_{x,\omega}(L)}{\overline{B}(L)} \right|\right],
\end{eqnarray}
where in the last inequality $\delta_1=\delta_1(\epsilon_1)$ will be chosen to be small enough.
The first sum is bounded as follows. First, notice that 
$\overline{B}_{x,\omega}(L)\leq E^\kappa[e^{\beta\mathbb{E}[\omega] T_L}\,]= e^{\kappa L} E[e^{-\lambda T_L} ] \leq e^{\kappa L}$. We can now use the fact that $\overline{B}(L)\geq \mu_0e^{-\overline{m}_B^aL}$ together with
 Jensen's inequality (employ also the fact that $\overline{B}_{x,\omega}(L) e^{-\kappa L}$ is less than one) to obtain
\begin{eqnarray*}
\left|\frac{1}{L}\log \sum_{x\in \mathcal{H}_{nL}} \hat\mu_{nL,\omega}^{\beta,\lambda}(x)\frac{ \overline{B}_{x,\omega}(L)}{\overline{B}(L)} \right|
&\leq&\left|\frac{1}{L}\log \sum_{x\in \mathcal{H}_{nL}} \hat\mu_{nL,\omega}^{\beta,\lambda}(x) \overline{B}_{x,\omega}(L) e^{-\kappa L}\right|+
\left|\frac{1}{L}\log   \overline{B}(L)  e^{-\kappa L} \right|\\
&\leq &
-\sum_{x\in \mathcal{H}_{nL}} \hat\mu_{nL,\omega}^{\beta,\lambda}(x)\frac{1}{L}\log \overline{B}_{x,\omega}(L)
+2\overline{m}_B^a+2\kappa.
\end{eqnarray*}
Taking first the conditional expectation in the last inequality, conditioned on $(\omega_x)_{\{x\colon x^{(1)}<nL\}}$, we obtain that the expectation of the right hand side of the last inequality can be bounded by 
\begin{eqnarray*}
-\frac{1}{L} \mathbb{E}\log \overline{B}_{\omega}(L)+2\overline{m}_B^a+2\kappa \leq 2( \overline{m}_B^q+\overline{m}_B^a+\kappa),
\end{eqnarray*}
 when $L$ is large enough.
The first sum in \eqref{Cesaro} is then bounded by 
\begin{eqnarray}\label{Cesaro1}
\frac{8\delta_1}{\epsilon_1}(\overline{m}_B^q+\overline{m}_B^a+\kappa).
\end{eqnarray}
To estimate the second summation in \eqref{Cesaro} we first choose $\delta_2=\delta_2(\epsilon_1)$ small enough. We also denote by $\mathcal{C}_{n,L,\delta_2}$ the event that 
$|L^{-1}\log \sum_{x\in \mathcal{H}_{nL}} \hat\mu_{nL,\omega}^{\beta,\lambda}(x)\frac{ \overline{B}_{x,\omega}(L)}{\overline{B}(L)}|>\delta_2$. We then estimate the expectation in \eqref{Cesaro} as follows
\begin{align}\label{Cesaro2}
 \mathbb{E}\left[ \left|\frac{1}{L}\log \sum_{x\in \mathcal{H}_{nL}} \hat\mu_{nL,\omega}^{\beta,\lambda}(x)\frac{ \overline{B}_{x,\omega}(L)}{\overline{B}(L)} \right|\right]
 &\leq
 \delta_2 +
  \mathbb{E}\left[ \left|\frac{1}{L}\log \sum_{x\in \mathcal{H}_{nL}} \hat\mu_{nL,\omega}^{\beta,\lambda}(x)\frac{ \overline{B}_{x,\omega}(L)}{\overline{B}(L)} \right|; \,\mathcal{C}_{n,L,\delta_2} \right]\nonumber\\
  &\leq
  \delta_2 +
  \mathbb{E}\left[ \left|\frac{1}{L}\log \sum_{x\in \mathcal{H}_{nL}} \hat\mu_{nL,\omega}^{\beta,\lambda}(x)\frac{ \overline{B}_{x,\omega}(L)}{\overline{B}(L)} \right|^2\right]^{1/2}\mathbb{P}( \mathcal{C}_{n,L,\delta_2} )^{1/2}.
\end{align}
To proceed further we need an a priori bound on the last expectation, which is independent of $n$. This is as follows. For each  $x\in\mathcal{H}_{nL}$ choose the sequence of points $x_0:=x$, $x_i:=x+i\hat{e}_1$, for $i=1,2,\dots,L$. We have that
\begin{eqnarray*}
\overline{B}_{x,\omega}(L)\geq \prod_{i=1}^L \overline{B}_{x_{i-1},\omega}(1;x_i),
\end{eqnarray*}
where $\overline{B}_{x_{i-1},\omega}(1;x_i)$ denotes the bridge of span $1$, starting from $x_{i-1}$ and ending at $x_i$. Using this, Jensen's inequality and the fact that $\overline{B}(L)\geq \mu_0e^{-\overline{m}_B^aL}$, in the same fashion as the route to obtaining \eqref{Cesaro1}, we have that
\begin{eqnarray*}
\left|\frac{1}{L}\log \sum_{x\in \mathcal{H}_{nL}} \hat\mu_{nL,\omega}^{\beta,\lambda}(x)\frac{ \overline{B}_{x,\omega}(L)}{\overline{B}(L)} \right|^2
&\leq &
2\sum_{x\in \mathcal{H}_{nL}} \hat\mu_{nL,\omega}^{\beta,\lambda}(x) \left|\frac{1}{L} \sum_{i=1}^L \log   \overline{B}_{x_{i-1},\omega}(1;x_i)e^{-\kappa }\right|^2\\
&&\qquad\qquad+2(2 \overline{m}_B^a+\kappa)^2\\
&\leq&
2\sum_{x\in \mathcal{H}_{nL}} \hat\mu_{nL,\omega}^{\beta,\lambda}(x) \frac{1}{L} \sum_{i=1}^L\left| \log   \overline{B}_{x_{i-1},\omega}(1;x_i)e^{-\kappa }\right|^2\\
&&\qquad\qquad+2(2 \overline{m}_B^a+\kappa)^2,
\end{eqnarray*}
and to estimate the expectation in \eqref{Cesaro2} we first take the conditional expectation conditioned on $(\omega_x)_{\{x\colon x^{(1)}<nL\}}$ leading to the bound
\begin{align}\label{log2}
 \mathbb{E}\left[ \left|\frac{1}{L}\log \sum_{x\in \mathcal{H}_{nL}} \hat\mu_{nL,\omega}^{\beta,\lambda}(x)\frac{ \overline{B}_{x,\omega}(L)}{\overline{B}(L)} \right|^2\right]
 &\leq
 2\mathbb{E} [ \left| \log   \overline{B}_\omega(1;0,\hat{e}_1) e^{-\kappa }\right|^2 ] +2(2 \overline{m}_B^a+\kappa)^2\nonumber\\
 &:=C^*_{\beta,\lambda}.
\end{align}
Next we control the probability $\mathbb{P}( \mathcal{C}_{n,L,\delta_2} )$ as follows
\begin{eqnarray*}
\mathbb{P}( \mathcal{C}_{n,L,\delta_2} )&=& \mathbb{P}\left( \sum_{x\in \mathcal{H}_{nL}}  \hat\mu_{nL,\omega}^{\beta,\lambda}(x) \overline{B}_{x,\omega}(L) >e^{\delta_2 L}\overline{B}(L)\right)
+\mathbb{P}\left( \sum_{x\in \mathcal{H}_{nL}}  \hat\mu_{nL,\omega}^{\beta,\lambda}(x) \overline{B}_{x,\omega}(L) <e^{-\delta_2 L}\overline{B}(L)\right),
\end{eqnarray*}
denoting $\tilde{B}_{x,\omega}(L):=\overline{B}_{x,\omega}(L)-\overline{B}(L)$ and using Chebyshev's inequality we have
\begin{eqnarray}\label{prob est}
\mathbb{P}( \mathcal{C}_{n,L,\delta_2} ) &\leq&\frac{1}{\overline{B}(L)^2} \left(\frac{1}{(e^{\delta_2 L}-1)^2}+\frac{1}{(1-e^{-\delta_2 L})^2}\right) \times\nonumber\\
&&\qquad\times\mathbb{E}\left[ \left( \sum_{x\in \mathcal{H}_{nL}}  \hat\mu_{nL,\omega}^{\beta,\lambda}(x) \tilde{B}_{x,\omega}(L)\right)^2\right].
\end{eqnarray}
To estimate the expectation in \eqref{prob est} we write 
\begin{eqnarray*}
\tilde{B}_{x,\omega}(L)&=&E^\kappa_x\left[ ( e^{-\beta\sum_y \overline\omega_y\ell_{T_L}(y)} -e^{-\overline\Phi_\beta(T_L)} );\text{Br}(L), \sup_{n\leq T_L}|X_n-x|\leq C_5 L\right]\\
&&+E^\kappa_x\left[ ( e^{-\beta\sum_y \overline\omega_y\ell_{T_L}(y)} -e^{-\overline\Phi_\beta(T_L)} );\text{Br}(L), \sup_{n\leq T_L}|X_n-x|> C_5 L\right],
\end{eqnarray*}
where $C_5$ is a large constant. We denote the first term above by $\tilde{B}^{loc}_{x,\omega}(L)$ and we note that it satisfies $\mathbb{E}[\tilde{B}^{loc}_{x,\omega}(L)]=0$. Moreover the second term is bounded by 
\begin{eqnarray*}
E^\kappa_x\left[ e^{\beta\mathbb{E}[\omega] T_L};\, \sup_{n\leq T_L}|X_n-x|> C_5 L\right]
  &\leq& E^\kappa\left[ e^{\beta\mathbb{E}[\omega] T_L}; T_L>C_5 L\right]\\
  &=& e^{\kappa L} E\left[ e^{-\lambda T_L}; \,T_L>C_5L\right]\\
  &\leq& e^{-\frac{\lambda}{2}C_5 L+\kappa L}  E\left[ e^{-\frac{\lambda}{2}T_L}\right]\\
  &\leq& e^{-\frac{\lambda}{4}C_5 L}
\end{eqnarray*}
provided that $C_5$ is chosen large enough. Notice also that $\tilde{B}^{loc}_{x,\omega}(L)$ and $\tilde{B}^{loc}_{y,\omega}(L)$ are independent when $|x-y|>2C_5L$. Using all these and expanding the square in the expectation of \eqref{prob est} we have that
\begin{eqnarray*}
&&\mathbb{E}\left[ \left( \sum_{x\in \mathcal{H}_{nL}}  \hat\mu_{nL,\omega}^{\beta,\lambda}(x) \tilde{B}_{x,\omega}(L)\right)^2\right]\\
&&\qquad=
\sum_{x,y\in \mathcal{H}_{nL}} \mathbb{E}\left[  \hat\mu_{nL,\omega}^{\beta,\lambda}(x) \hat\mu_{nL,\omega}^{\beta,\lambda}(y) \left( \tilde{B}^{loc}_{x,\omega}(L)  +O(e^{-\frac{\lambda C_5}{4} L})\right)  
\left( \tilde{B}^{loc}_{y,\omega}(L)  +O(e^{-\frac{\lambda C_5}{4} L})\right) \right]\\
&&\qquad=
\sum_{x,y\in \mathcal{H}_{nL}} \mathbb{E}\left[  \hat\mu_{nL,\omega}^{\beta,\lambda}(x) \hat\mu_{nL,\omega}^{\beta,\lambda}(y)\right] \,\,\mathbb{E}\left[ \tilde{B}^{loc}_{x,\omega}(L)  
\tilde{B}^{loc}_{y,\omega}(L) \right] + O(e^{-\frac{\lambda C_5}{8} L})\\
&&\qquad\leq 
e^{2\kappa L}\sum_{x,y\in \mathcal{H}_{nL}\colon |x-y|\leq 2C_5L} \mathbb{E}\left[  \hat\mu_{nL,\omega}^{\beta,\lambda}(x) \hat\mu_{nL,\omega}^{\beta,\lambda}(y)\right] +O(e^{-\frac{\lambda C_5}{8} L})\\
&&\qquad\leq
CC^{d-1}_5 L^{d-1}e^{2\kappa L} \,\mathbb{E}\left[  \sup_{x\in \mathcal{H}_{nL}}\hat\mu_{nL,\omega}^{\beta,\lambda}(x) \right] +O(e^{-\frac{\lambda C_5}{8} L}).
\end{eqnarray*}
Assuming that \eqref{localization} is false we have that $\limsup_{n\to\infty}\sup_{x\in \mathcal{H}_{nL}}\mu_{\omega,{nL}}(x)=0$, $\mathbb{P}-a.s$,  and inserting this into \eqref{prob est} we get that, for $n$ large enough, depending on $L,\delta_1,\delta_2$ and having chosen $C_5$ large enough so that also $2 \overline{m}_B^a-\lambda C_5/8<-\lambda C_5/16$, we have
\begin{eqnarray}\label{prob est2}
\mathbb{P}( \mathcal{C}_{n,L} ) \leq \delta_2^2,
\end{eqnarray}
This estimate together with \eqref{log2} inserted into \eqref{Cesaro2} leads to
\begin{eqnarray*}
 \frac{4}{\epsilon_1}\frac{1}{N}\sum_{n=\delta_1 N+1}^{N-1}\mathbb{E}\left[ \left|\frac{1}{L}\log \sum_{x\in \mathcal{H}_{nL}} \hat\mu_{nL,\omega}^{\beta,\lambda}(x)\frac{ \overline{B}_{x,\omega}(L)}{\overline{B}(L)} \right|\right]
 \leq 
  \frac{4 \delta_2}{\epsilon_1}\left(1 + \left( C_{\beta,\lambda}^*\right)^{1/2}  \right),
\end{eqnarray*}
which is of course valid if $N>N_0(\delta_1,\delta_2,L)$, large enough. This together with \eqref{Cesaro1} and \eqref{Cesaro} show that, if $\delta_1,\delta_2$ are chosen small enough, both depending on $\epsilon_1$, we have that 
\begin{eqnarray*}
\mathbb{P}\left(-\frac{1}{NL}\sum_{n=0}^{N-1}\log \sum_{x\in \mathcal{H}_{nL}} \hat\mu_{nL,\omega}^{\beta,\lambda}(x) \overline{B}_{x,\omega}(L)>\overline{m}_B^a +\frac{\epsilon_1}{2}\right) <\epsilon_1.
\end{eqnarray*}
Therefore, with probability greater than $1-\epsilon_1$ we have that
\begin{eqnarray}\label{last}
-\frac{1}{NL}\sum_{n=0}^{N-1}\log \sum_{x\in \mathcal{H}_{nL}} \hat\mu_{nL,\omega}^{\beta,\lambda}(x) \overline{B}_{x,\omega}(L)<\overline{m}_B^a +\frac{\epsilon_1}{2},
\end{eqnarray}
which leads to contradiction since by \eqref{heuristic}, the left hand side of \eqref{last} is larger than  $-\frac{1}{NL}\log \overline{B}_\omega(NL)$, which, for $N$ large, converges $a.s.$ to $\overline{m}_B^q>\overline{m}_B^a+\epsilon_1$.

\vskip 4mm 
{\bf Acknowledgement:} This work was motivated by discussions with D.Ioffe at the Institut Henri Poincar\'e - Centre Emile Borel and Technion-Haifa. I would like to thank D.Ioffe for several stimulating discussions and the institutes for the hospitality. I would also like to thank A.S. Sznitman for his constructive comments on an earlier draft of the paper. Part of this work was supported by IRG-246809.

\end{document}